\documentclass[11pt]{article}
\usepackage{amsfonts}
\usepackage{amssymb}
\usepackage{enumerate}
\usepackage{color, amsmath,amssymb, amsfonts, amstext,amsthm, latexsym}

\usepackage{amssymb, epsfig, amssymb, latexsym}
\usepackage{amsmath}
\usepackage{graphicx}
\usepackage{longtable}
\usepackage{hyperref}

\newtheorem{theorem}{Theorem}

 \newtheorem{lemma}{Lemma}
\newtheorem{definition}{Definition}
\newtheorem{remark}{Remark}
\newcommand{\om}{\omega}
\newcommand{\Om}{\Omega}
\renewcommand{\phi}{\varphi}

\renewcommand{\a}{\alpha}

\newcommand{\R}{{\mathbb R}}
   
   \newcommand{\EX}{{\Bbb{E}}}
   \newcommand{\PX}{{\Bbb{P}}}

\newcommand{\BE}{\begin{equation}}
\newcommand{\EE}{\end{equation}}
\newcommand{\BEN}{\begin{equation*}}
\newcommand{\EEN}{\end{equation*}}
\newcommand{\BAL}{\begin{align}}
\newcommand{\EAL}{\end{align}}
\newcommand{\BAN}{\begin{align*}}
\newcommand{\EAN}{\end{align*}}

\def \th {\theta}
\def \F {\mathcal F}

\def \R {\mathbb R}
\def \L {\Lambda}

\def \l {\lambda}
\def \DD {\Delta}
\def \A {\mathcal {A}}
\def \Ga {\Gamma}
\def \M {\mathcal{M}}

\begin{document}

\title{Simulating Stochastic Inertial Manifolds\\ by a Backward-Forward Approach
\footnote{This work was partly supported by the NSF Grants 1025422,
0923111 and 0731201, by the US DOE grant DESC0002097, and by the  NSFC grants 10971225 and 11028102.} }
 \author{Xingye Kan$^{1, 2}$,   Jinqiao Duan$^{1,2}$, Ioannis G. Kevrekidis$^3$ and Anthony J. Roberts$^4$\\
 \\
1. Department of Applied Mathematics\\ Illinois Institute of Technology \\
   Chicago, IL 60616, USA \\
   E-mail: xkan@hawk.iit.edu, \;\; duan@iit.edu   \\\\
2. Institute for Pure and Applied Mathematics, UCLA\\
Los Angeles, CA 90095, USA\\\\
3. Program in Applied and Computational Mathematics\\
$\&$\\
Department of Chemical and Biological Engineering\\
Princeton University\\
    Princeton, NJ 08544, USA  \\
    E-mail: yannis@princeton.edu  \\\\
4. School of Mathematical Sciences\\University of Adelaide\\Adelaide, Australia\\
    E-mail:  anthony.roberts@adelaide.edu.au  }

\date{June 18, 2012}

\maketitle

\begin{abstract}
\noindent A numerical approach for the approximation of inertial
manifolds of stochastic evolutionary equations with multiplicative
noise is presented and illustrated. After splitting the stochastic evolutionary
equations into a backward and a forward part, a numerical scheme is devised for
solving this backward-forward stochastic system, and an ensemble of graphs representing the inertial manifold is consequently obtained. This numerical approach is tested in two illustrative examples: one is for a system of stochastic ordinary differential equations and the other is for a stochastic partial differential equation.

\bigskip

\paragraph {Key words}
Invariant manifolds, inertial manifolds, random dynamical systems,
backward and forward stochastic differential equations, stochastic partial differential equations, numerical schemes
\end{abstract}

\section{Introduction}

%The first natural tool of studying the dynamics of dissipative
%evolutionary equations is the concept of universal attractor \cite{Titi}.
The concept of inertial manifolds for deterministic partial
differential equations was  introduced in 1980s \cite{Constantin,
Foias}. These manifolds are finite dimensional invariant   manifolds
that attract every trajectory at an exponential rate. They play an
important role in the study of the long-time behavior of solutions,
since through them the dynamics of a large system can be described
by a finite dimensional system. To be more specific, the dimension
of the state space is reduced by projecting the system onto the
inertial manifold once the existence of the manifold has been
proved. For certain dissipative nonlinear systems, the inertial
manifold can be shown to exist and it exponentially attracts
solution orbits \cite{Foias}. In the case where the existence of the
inertial manifold is unknown, approximate inertial manifolds are
introduced \cite{Foias3, Titi, Titi2, Marion}. Several authors
considered the construction of approximate inertial
manifolds \cite{Titi2} as well as numerical simulations using these
manifolds \cite{Foias2, Graham, Jolly}.

Inertial manifolds have also been  considered for
  stochastic systems \cite{BenFlandoli, Chueshov, DaPrato, DuanLuSchm, DuanLuSchm2}. In one such study, Da Prato and Debussche \cite{DaPrato}
introduced the concept of a stochastic inertial manifold for an
abstract stochastic evolutionary equation in a Hilbert space $H$
(with scalar product $\langle\cdot, \cdot\rangle$ and the induced
distance $d(\cdot, \cdot)$)

\begin{align}\label{eq1}
\left\{
\begin{array}{rl}
&du + (Au + f(u))dt = g(u)dW_t,\\
&u(0) = u_0,
\end{array}
\right.
\end{align}
where $A$ is a   linear operator, $f$ and $g$ are two nonlinear
functions, and $W_t$ is a Wiener process taking values in another Hilbert space $U$.

%Let $(S(t,s;\om)_{t\geq s, \om\in\Om}$ be the solution mapping for
%the stochastic system~\eqref{eq1}, i.e. for $x$ in $H$,
%$S(t,s;\om)x$ is the value at time $t$ of the solution of
%\eqref{eq1} which is equal to $x$ at time $s$.
% A stochastic inertial manifold for \eqref{eq1} is a  random family of
%manifolds $(M(\om))_{\om\in\Om}$ such that
%\begin{enumerate}[(i)]
% \item it is invariant in
%the sense that
%\[
%S(t,s;\om)M(\theta_s\om) = M(\theta_t\om),\quad t\geq s,
%\]
%where $\theta_t $  is the Wiener shift (see Section 2 below) associated with the
%Wiener process $W_t$, and \\
%\item it exponentially  attracts all solution orbits. Namely
%\[
%d(S(t,s;\om)x, M(\theta_t\om))\rightarrow0,\quad \text{for any}\quad
%x\in H,
%\]
%exponentially as $t\rightarrow\infty.$
%\end{enumerate}

Although a  theoretical framework for stochastic inertial manifolds
has been set up \cite{BenFlandoli, DuanLuSchm, DuanLuSchm2}, it is
desirable to  efficiently approximate  stochastic inertial manifolds
and perform simulations on them. Recently, Roberts introduced a
normal form transformation of stochastic differential systems when
the dynamics contains both slow modes and quickly decaying modes
\cite{Roberts},
%@techreport{Roberts07d,
    %Author = {A.~J. Roberts},
    %Institution = {\url{http://www.maths.adelaide.edu.au/anthony.roberts/sdenf.php}},
    %Title = {Normal form of stochastic or deterministic multiscale differential equations},
    %Year = {2009}}
in which algebraic techniques were used. This is intended for reduction of
finite dimensional systems and it might be used to test the
numerical techniques that are proposed for the infinite dimensional
case and also to inspire future development.

In this paper, we   introduce a numerical scheme for simulating the
stochastic inertial manifold of a stochastic evolutionary system
with multiplicative noise, which includes stochastic differential equations (\textsc{sde}s)
and stochastic partial differential equations (\textsc{spde}s). By projecting to
a countable basis, an \textsc{spde} could be converted to an infinite dimensional
system of \textsc{sde}s. The main idea is to solve a coupled backward-forward
system of \textsc{sde}s where
the backward part is finite dimensional, but the forward part is either infinite or high dimensional. The forward Picard type iteration scheme
\cite{Bender, Bender2, Douglas} is performed on the backward part, and an Euler discretization
scheme is applied to the forward part. The graph for the stochastic inertial manifold is consequently
obtained. Two examples, one for a system of \textsc{sde}s and one for an \textsc{spde},
are presented to illustrate our backward-forward approach.

The rest of this paper is organized as follows. Section 2 formulates
the problem and briefly reviews the analytical results
\cite{DaPrato}. Section 3 introduces the numerical scheme and
performs the error analysis. An approximation procedure is discussed
in Section 4. Finally two numerical examples are presented in
Section 5.

\section{Problem formulation}

We consider a  stochastic evolutionary system in a Hilbert space $H$ of the form

\begin{equation}\label{eq2}
du + Au\,dt = f(u)dt+g(u)dW_t,
\end{equation}
where the following conditions hold:

 \begin{enumerate}

\item  \emph{Linear part.} \\
$A$ is a self-adjoint operator in $H$ with
eigenvalues
\[
\l_1\leq\l_{2}\leq\cdots\leq\l_{k}\leq\cdots\leq\l_{j}\rightarrow+\infty.
\]

\item   \emph{Nonlinear part.}\\
$f: D(A^{\a})\mapsto H$ and $g: D(A^{\a})\mapsto L_2^0(D(A^\beta))$, for some $\a, \beta\in[0,1)$,
 are globally Lipschitz and bounded, with Lipschitz constants $L_f$ and~$L_g$ respectively,
i.e., for any $u,v\in D(A^\a)$,
\begin{align}\label{eq7}
\mid f(u)-f(v)\mid_H\leq L_f\mid u-v\mid_{D(A^\a)},\notag\\
\mid g(u)-g(v)\mid_H\leq L_g\mid u-v\mid_{D(A^\a)}.
\end{align}

When $f$ and $g$ are only locally Lipschitz, but the corresponding
deterministic system has a bounded absorbing set in an appropriate
state space, it is possible to cut-off $f$ and $g$ to zero outside a
ball containing the absorbing set. In this case the modified
(``cut-off") system has  globally Lipschitz drift and noise
intensity.

\item   \emph{Noise part.}\\
Although the Wiener process $W_t$ may take values in a Hilbert
space~$U$, to be specific, here we just consider $W_t$ to be a
two-sided one dimensional Wiener process, defined in a probability
space $(\Om,\F,\mathbb P)$, and adapted to a filtration $\F_t,
t\in\R$; note that $\F_t$ is  a two-parameter filtration
\cite{Arnold} or a one-parameter filtration starting at time
$-\infty$ instead of at time $0$ \cite{deSam}. More precisely, for
each $ t\in\R$,
\begin{align}\label{eqfiltration}
&\F_{-\infty}^t: = \vee_{s\leq t}\F_s^t,    \\
&\F_s^t := \sigma(W(u), s\leq u\leq t), \nonumber
\end{align}
which is the information generated by the Wiener process W on
the interval $(-\infty, t]$. We denote $\F_{-\infty}^t$ as $\F_t$
for simplicity here and henceforth.

The two-sided Wiener process $W_t$ is defined in terms of two
independent Wiener processes $\hat{W}_t$ and $\tilde{W}_t$ ($t\geq
0$), as follows,

\begin{equation*}
  W_t = \begin{cases}
         \hat{W}_t,   &\text{if $t \geq 0$,}\\
         \tilde{W}_{-t},    &\text{if $t<0$.}
    \end{cases}
\end{equation*}
The adaptedness means $W_t$ is measurable with respect to $\F_t$ for
each~ $t$. %In this case, the noise intensity coefficient $g$ may be
%a (locally) Lipschitz nonlinear function.

\end{enumerate}

\begin{remark}\label{rmk1}
The two-parameter filtration defined in \eqref{eqfiltration}
requiring $\F_s^t\subset\F_{s'}^{t'}$ for $s'\leq s\leq t\leq t'$ is
consistent with the well-known filtration for positive time. Indeed,
$\{\F_t\}_{t\geq0}$ is $\{\F_0^t\}_{t\geq0}$ in the two-parameter
setting. The only difference between one-parameter and two-parameter
filtrations in the above setting is their starting time. Since the
filtration specifies how the information is revealed in time, the
property that a filtration is increasing corresponds to the fact the
information is not forgotten. However, the generalization of
filtration from one-parameter to two-parameter, while maintaining
the property that a filtration is increasing, results in technical
difficulties in constructing invariant manifolds since a backward
\textsc{sde} is encountered. Overcoming this difficulty will be
discussed in detail in   Remark \ref{rmk3}.
\end{remark}

As   discussed in  \cite{Arnold, DuanLuSchm, DuanLuSchm2, Duan},  it
is appropriate and convenient to consider the canonical sample
space, by identifying sample paths of the Wiener process $W_t$ with
continuous curves (passing through the origin at $t=0$ since
$W_0=0$). Namely, a sample path is now a \emph{point}  in the space
$C(\R, \R)$ of continuous functions: $W_t(\om)= \om(t)$. Therefore
the sample space is taken to be
\[
\Om = \{\om\in C(\R;\R):\om(0)=0\}
\]
and $\mathbb P$ is taken to be the Wiener measure. This is analogous
to the situation of dice-tossing, where we take six face values, $1,
2, 3, 4, 5, $ and $6$, as   samples in the \emph{canonical} sample
space $\Om =\{1, 2, 3, 4, 5, 6 \}$.  When we ``toss" a Wiener
process $W_t$, we see continuous (but nowhere differentiable) curves
as ``face values" or samples.

%Since for fixed $\om$, say $\check{\om}$, $\check{\om}(t)$ is a continuous sample
%path; while for fixed $t$, every $\om$ is evaluated at $t$, and thus
%has a value in $\R$, which is the same as fixing $t$ for
%$W(t)(\om)$ yields a random variable taking values in $\R$.\\

  The Wiener shift $\theta_t$ is defined as a mapping in the canonical sample space $\Om$, for each fixed $t\in \R$,
\begin{align}
\theta_t:\ &\Om\rightarrow\Om   \nonumber  \\
&\om\mapsto\check{\om}\quad \text{such that}\quad
 \theta_t\om(s) = \check{\om}(s) \triangleq \om(t+s) - \om(t),\quad s,t\in\R.  \label{theta}
\end{align}

\begin{remark}\label{wienershift}
The Wiener shift defined in \eqref{theta} is a measure preserving
transformation, i.e.
\begin{align*}
\mathbb P(A) = \mathbb P(\theta_t^{-1}A), \forall A\in\F,
\end{align*}
\end{remark}
\noindent where $\mathbb P$ is the Wiener measure and $\mathbb P(A) = \mathbb
P(\theta_tA)$ is implied.\\

By a simple calculation, we see that $\theta_0=\operatorname{Id}$
(the identity mapping in $\Om$) and $\theta_{s+t}=\theta_s\circ
\theta_t$. Hence the Wiener shift is a deterministic dynamical
system (or a flow) in $\Om$. The above equation \eqref{theta} means
that
\begin{equation}\label{ddd}
   W_s(\theta_t \om) = W_{t+s}(\om)- W_t(\om) \thickapprox dW_t(\om).
\end{equation}
Thus $\theta_t$ is closely related to the noise in the stochastic
system \eqref{eq2} and is often called the \emph{driving flow}. The
solution mapping satisfies the property \cite{DaPrato}
\begin{equation}
   S(t, s; \om)x = S(t-s, 0; \theta_s\om)x, \; x\in H, \; t \geq s, \; \PX \;  a.s.
\end{equation}

\begin{definition} A stochastic inertial manifold for \eqref{eq2} is a random
family of manifolds $\{\mathcal {M}(\om)\}_{\om\in\Om}$ which is measurable with respect to
$\mathcal F_0$ and satisfies the following three
properties:
\begin{enumerate}
\item  Each realization of  $ \mathcal {M}(\om) $  is a deterministic manifold:  $\mathcal M(\om)$ is a Lipschitz (or smooth) manifold for
$\mathbb{P}-$almost all $\om\in\Om$.
\item Invariance:
\[
S(t,0,\om)\mathcal M(\om)=\mathcal M(\theta_t\om),\quad t\in\mathbb
R^+, \quad   \quad \mathbb P\text{ a.s.}
\]
\item Exponential attraction: for any $x\in H$
\[
\lim_{t\rightarrow\infty}d(S(t,0;\om)x,\mathcal M(\th_t\om))=0,\quad
\text{exponentially in }\ L^2(\Om),
\]
\end{enumerate}
where $u(t,s;\om)$ is the unique solution for \eqref{eq2} defined
for $t\in [s, +\infty)$ such that
\[
u(s,s;\om)=u_s
\]
and $S(t,s;\om)$ is the solution mapping
\[
u(t,s;\om) \triangleq S(t,s;\om)u_s(\om).
\]
%  For any $t\geq s$, the mapping $u(t,s;\cdot)\rightarrow
%S(t,s;\cdot)u_s(\cdot)$ is continuous from $L^2(\Om)$ into $L^2(\Om)$.
\end{definition}

\begin{remark}
\begin{enumerate}
\item
$\mathcal M(\om)$ is a ``random variable" mapping all the samples
(continuous functions) in the canonical sample space $\Om$ into the
space of deterministic Lipschitz (or smooth) manifolds
\cite{Marsden}, and thus is measurable with respect to $\mathcal
F_0$.
Figure~\ref{stoimf} explains the meaning of random family of
manifolds $\{\mathcal M(\om)\}_{\om\in\Om}$ heuristically.

\item
The stochastic invariant property can be considered as a natural
generalization of invariance property in the deterministic setting.
In the deterministic case, $S(t,0)\mathcal M=\mathcal M,\quad
t\in\mathbb R^+$, which can be seen as $S(t,0)\mathcal
M(\om)=\mathcal M(\om)$ since there is one and only one sample in
such setting. In the stochastic setting, the invariance is in the
sense of probability. More precisely, under the system evolution or
solution mapping~$S$, if we start from somewhere on the manifold
$\mathcal M(\om)$, then after time $t$, we will stand on manifold
$\mathcal M(\theta_t\om)$ whose likelihood  of occurrence is the
same as $\mathcal M(\om)$. This is implied by the measure-preserving
property of Wiener shift $\theta_t$, i.e. $\mathbb
P(\{\om\})=\mathbb P(\{\theta_t\om\})$, see
Remark~\ref{wienershift}. From now on, we do not distinguish
``$\om$" and ``$\theta_t\om$" when we say ``a fixed sample".

\end{enumerate}

\begin{figure}[!htb]
\centerline{
\scalebox{0.5}{
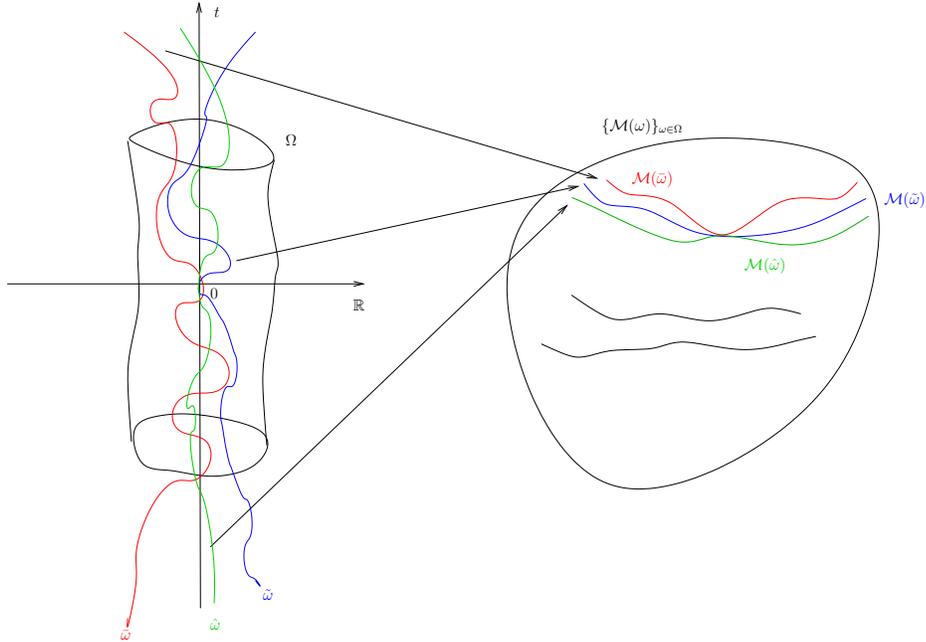}}
\caption{\small \sl Sketch of a stochastic inertial manifold: Three random samples ($\omega, \hat{\omega}, \tilde{\omega}$) on the left, and the corresponding three realizations ($\M(\omega), \M(\hat{\omega}), \M(\tilde{\omega})$) of the manifold on the right.
} \label{stoimf}
\end{figure}

% By this definition, for $\mathbb{P}-$almost all
%fixed  sample $\om\in\Om$, $\mathcal M(\om)$ is a deterministic
%Lipschitz (or smooth) manifold \cite{Marsden}. The stochastic
%invariance property says that under the system evolution or solution
%mapping $S$, the ensemble of these deterministic manifolds does not
%change over time.

\end{remark}

\bigskip

Recall that the   construction of an inertial manifold in the
deterministic case amounts to finding a graph above an eigenspace of
the linear operator $A$. By analogy, we take a projection $P$ to a
finite dimensional eigenspace $H^+$, and look for a function $\Phi$ from
$H^+=PH$ to $H^-=(I-P)H$ whose graph is invariant under the evolution of the
stochastic system \eqref{eq2}. The analytical foundation of our
numerical method is the combination of two well known methods
for constructing deterministic inertial manifolds. One is the
Lyapunov--Perron method, and the other one is the graph transform
method \cite{Constantin}. Roughly speaking, the Lyapunov--Perron method looks for
solutions of the original equation whose components on $(I-P)H$ are
bounded for negative time. The graph transform method is to let an inertial manifold (typically, the flat manifold $PH$) evolve under the system evolution and to
verify that the image of $PH$ at time $t$ is a graph which will
converge to an invariant manifold as $t\rightarrow\infty$.

As introduced by Da Prato and Debussche \cite{DaPrato},  we
reformulate \eqref{eq2} into a backward part and a forward part for
time $t \in [-T, 0]$:
\[  u:= X+Y, \]
and
\begin{align}\label{eq3333}
\left\{
\begin{array}{rl}
&dX + AX\,dt = Pf(X+Y)dt + Pg(X+Y)dW_t,\\
&X(0) = X_0,\\
&dY + AY\,dt = Qf(X+Y)dt + Qg(X+Y)dW_t,\\
&Y(-T) = 0,
\end{array}
\right.
\end{align}
where $P$ is a  projection from $H$ to the eigenspace spanned by the
first $k$ eigenvalues $\l_1,\cdots,\l_k$ of $A$, and $Q:=I-P$. Here
$k$ is determined by  the eigenvalues and Lipschitz constants for
the existence of stochastic inertial manifolds \cite{BenFlandoli,
DaPrato}. Hence $X$ is of finite dimension $k$ and $Y$ is of
infinite dimension.

   The problem \eqref{eq3333} involves a
backward stochastic evolutionary equation.

\begin{remark}\label{rmk3}
At first glance, we might want to solve the equations \eqref{eq3333}
whose unknown is the pair $(X,Y)$, in the interval $[-T,0]$.
%\begin{align*}
%\left\{
%\begin{array}{rl}
%&dX + AX\,dt = Pf(X+Y)dt + Pg(X+Y)dW_t,\\
%&X(0) = X_0,\\
%&dY + AY\,dt = Qf(X+Y)dt + Qg(X+Y)dW_t,\\
%&Y(-T) = 0.
%\end{array}
%\right.
%\end{align*}
However, this type of problem does not have solutions in general.
For the existence and uniqueness of solution of \textsc{sde}s, in
addition to the usual requirement as in the case of \textsc{ode}s, it is also
necessary for the solution to be adapted to the filtration generated
by the noise. Since the filtration is a collection of fields,~
$\mathbb{F}:=\{\ldots,\F_{-T},\F_{-t},\F_0,\F_{t},\F_{T},\ldots\}$,
$\F_t\subset\F_{t+1}$, if we use the usual backward integration
method, i.e., finding the solution at time $t$ by making use of the
solution at time $t+1$, then this $t-$solution is
$\F_{t+1}-$measurable but not necessarily $\F_t-$measurable, which
violates the definition of solution for \textsc{sde}s.

%Suppose for a given $\om^*$, we can solve the forward equation, but
%$X_0(\om^*)$ is not available unless $X_0$ is a constant, thus the
%backward equation is not solvable. This is because when we plot the
%solution of a \textsc{sde}, we are in fact plotting one sample path
%of the solution, and this sample is only conceptually, namely, we
%cannot recognize which sample it is but only assure it is one sample
%in $\Om$.

\end{remark}

In order to overcome this difficulty, the terminal value problem of
\textsc{sde} is reformulated in such a  way as to allow a
solution that is $\{\F_t\}_{t\leq0}-$adapted.

By Proposition 3.1 introduced by Da Prato and Debussche
\cite{DaPrato}, for every $X_0$ that is $\mathcal F_0-$measurable
and square integrable, there exists a unique triple $(X,Y,M_t)$ such
that
\begin{enumerate}
\item  $X:[-T,0]\mapsto PH $ is mean-square continuous and adapted,\\
\item $Y:[-T,0]\mapsto QD(A^\a)$ is mean-square continuous and adapted,\\
\item  $M_t$ is a square integrable martingale with values in $PH$,
\end{enumerate}
and $(X,Y,M_t)$ solves the following combined backward-forward stochastic system,  for time $t \in [-T, 0]$
\begin{align}\label{eq4444}
\left\{
\begin{array}{rl}
&dX + AX\,dt = Pf(X+Y)dt + dM_t,\\
&X(0) = X_0,\\
&dY + AY\,dt = Qf(X+Y)dt + Qg(X+Y)dW_t,\\
&Y(-T) = 0.
\end{array}
\right.
\end{align}
 %The system \eqref{eq4444} differs from system \eqref{eq3333} in that $Pg(X+Y)dW_t$ is now replaced by $dM_t$.
where
\[
M_t = X_t - X_0 - \int_t^0AX_sds + \int_{t}^0Pf(X_s + Y_s)ds.
\]

Note that the solution of this system \eqref{eq4444}   sits on the
interval $[-T, 0]$, the first $k$ components of $u$,  i.e. $X$,
travel backward from 0 to $-T$, and the remaining infinitely many
components of $u$,
i.e. $Y$, travel forward from $-T$ to 0. %The reason for the initial

Figure \ref{schematicbf} is a schematic illustration of this
backward-forward method.

\begin{figure}[!htb]
\centerline{
%\begin{center}
%\vspace{10mm}
\scalebox{0.8}{
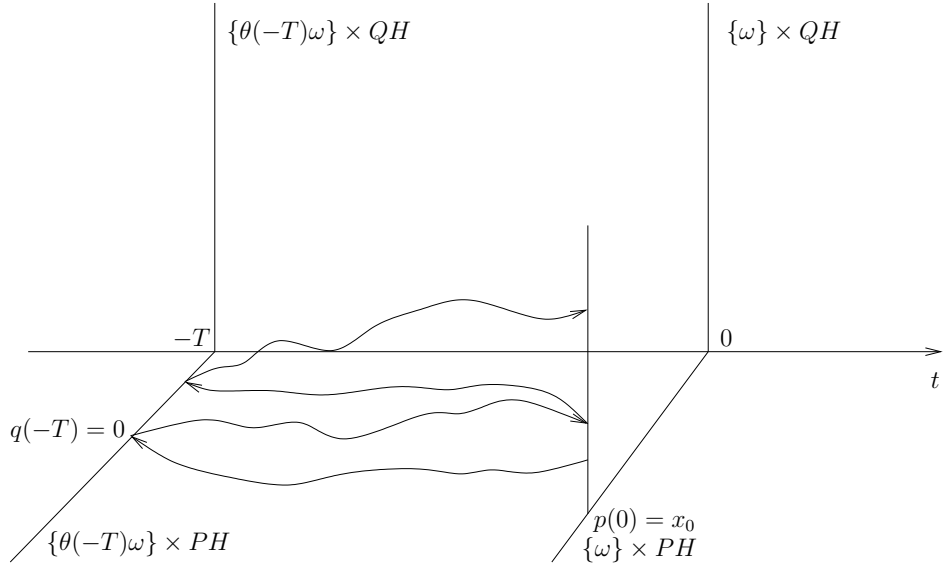}}
%\end{center}
\caption{\small \sl Schematic illustration of the backward-forward
approach for approximating a stochastic inertial manifold.}
\label{schematicbf}
\end{figure}

%\begin{figure}
%\centering
% \includegraphics[scale=0.7]{bf3.pdf}
%\includegraphics[bb=0 0 3000 2400,width=10cm,keepaspectratio]{Wuhan.jpg}
%\caption{\small \sl Schematic illustration of backward-forward
%approach for finding stochastic inertial
%manifold}\label{schematicbf}
%\end{figure}

For each fixed $T>0$, define a mapping $\Phi_T$  via
\begin{equation}\label{mfd222}
Y(0,\om):=\Phi_T(X_0(\om),\om),\ \om \in\Om.
\end{equation}
Da Prato and Debussche \cite{DaPrato} showed that the limit of
$(\Phi_T)_{T>0}$ as $T\rightarrow\infty$ in $L^2(\Om)$ is the
function $\Phi$ whose graph above $PH$ is the inertial manifold
\[
\mathcal M(\om)=\{u=X+Y:\;  Y=\Phi(X,\om), X\in PH \}.
\]
In the next section we discuss how to numerically simulate
\eqref{eq3333} and then in Section 4 we approximate the mapping
$\Phi$ and thus obtain an approximate inertial manifold $\mathcal
M$.

\section{A numerical scheme}

In this section, we devise a numerical scheme to compute the
solution of the backward-forward stochastic system \eqref{eq3333}.
Our scheme is inspired by \cite{Bender}. Some related references are
\cite{Bender2, Douglas}.

\subsection{Main idea}

Literally speaking, the backward-forward numerical iteration scheme for \eqref{eq4444}
works as follows: We start the first iteration by setting $Y$ to be
zero (flat manifold) for the entire time interval $[-T,0]$, and $X$ to be zero on
 $[-T,0)$ together with the terminal $X$ value to be any $\F_0$-measurable random variable. We
obtain the $X$ trajectory backward in time, by using future
information of $(X, Y)$ at the previous iteration, and then we generate the
$Y$ trajectory future in time using past in time information of $(X, Y)$
at the previous iteration.
%Then we use future $X$ and $Y$ trajectories  and the terminal $X$
%value to find $X$ values backward in time (i.e., solve the
%\textsc{sde} for $X$ backward in time). For the next iteration, we
%compute $Y$ values (i.e., solve the \textsc{sde} for $Y$ forward in
%time) by using the previously available $X$ and the current $Y$
%value, and then compute $X$ value as in the first iteration and so
%on.
The iteration is stopped when the distance of two consecutive $(X, Y)$
trajectories is less than a preset tolerance; we then use the
terminal $Y$ value at that iteration to approximate one point on the
manifold $\Phi$. %These $Y$ values at $t=0$ are used to calculate
%$\Phi_T$, which provides an approximation for the inertial manifold
%$Y=\Phi(X)$.
%Since the convergence of $X$ value is
%guaranteed by similar proof as in \cite{Bender}, the convergence
%rate of $Y$ value depends only on $\Delta t$.
We will illustrate the
method in more detail in Section~\ref{sec4}.

Before getting to the backward-forward approach, we need the
following preparatory work.

 Let $\{e_1, e_2, \ldots\}$ be an orthonormal basis for $H$. In
numerical analysis, we approximate the Hilbert space (could be
infinite dimensional) by  a $(k+l)-$dimensional subspace, and
project \eqref{eq2} into this subspace, i.e.
\begin{align}
du^{k+l} + A_{k+l}u^{k+l}\,dt = f_{k+l}(u^{k+l})dt+g_{k+l}(u^{k+l})dW.
\end{align}
In theory, $l$ could be infinite or a big natural number. The above projection is usually done by the Galerkin method.
We denote the numerical approximations of   $X,Y$ by $x,y$, respectively. Then
$PH=\R^k$ and $QH=\R^l$,
\[ u^{k+l} = x+y, \; \; x\in \R^k \; \mbox{ and } \; y \in \R^l, \]
and
\begin{align}\label{eq3}
\left\{
\begin{array}{rl}
&dx + Sx\,dt = \sum_{i=1}^k\langle e_i,f(X+Y)\rangle dt + dM_t, \\
&x(0) = x_0,\\
&dy + Uy\,dt = \sum_{i=k+1}^{k+l}\langle e_i,f(X+Y)\rangle dt + \sum_{i=k+1}^{k+l}\langle e_i,g(X+Y)\rangle dW_t,\\
&y(-T) = 0,
\end{array}
\right.
\end{align}
where $Sx := \sum_{i=1}^k\langle e_i,A_{k+l}u^{k+l}\rangle, Uy:=\sum_{i=k+1}^{k+l}\langle e_i,A_{k+l}u^{k+l}\rangle$.

For each fixed $T>0$, define a mapping $\Phi_T$  via
\begin{equation}\label{mfd}
y(0,\om):=\Phi_T(x_0(\om),\om),\ \om\in\Om.
\end{equation}
The limit of $(\Phi_T)_{T>0}$ as $T\rightarrow\infty$ in $L^2(\Om)$
is the function $\Phi$, and its graph above
$\R^k$ is the (approximate) inertial manifold $\mathcal M$
\[
\mathcal M(\om) = \{u=x+y: \; y=\Phi(x,\om), \;\; x\in \R^k\}.
\]
Therefore, our goal of constructing the stochastic inertial manifold $ \mathcal{M}$ is to numerically solve
\eqref{eq3} in order to obtain $\Phi_T$ for a large $T>0$ and for a number of sample $\om$'s.

For convenience, we denote
\begin{align*}
&\sum_{i=1}^k\langle e_i,f(X+Y)\rangle := f_1(x,y),\\
&\sum_{i=k+1}^{k+l}\langle e_i,f(X+Y)\rangle := f_2(x,y),\\
&\sum_{i=k+1}^{k+l}\langle e_i,g(X+Y)\rangle := g_2(x,y)
\end{align*}
 for the rest of this
paper. In the following, we also denote  $x(t)$ by $x_t$ and $y(t)$
by $y_t$.

\bigskip

In principle, the solution of \eqref{eq3} can be obtained as the
limit of a Picard type iteration $(x^{(n)},y^{(n)})$ as introduced
by Da Prato and Debussche \cite{DaPrato}. To be more precise,
$(x_t^{(0)},y_t^{(0)})\equiv(0,0)$, and $(x_t^{(n)},y_t^{(n)})$ is
the solution of the following iteration scheme
\begin{align}\label{eq4}
\left\{
\begin{array}{rl}
 x_t^{(n)} &= \EX\bigg(x_0 + \int_t^0Sx_s^{(n-1)}ds -
\int_t^0f_1(x_s^{(n-1)},y_s^{(n-1)})ds\,\vline\,\F_t\bigg),\\
y_t^{(n)} &= e^{-U(t+T)}y(-T)+\int_{-T}^te^{-U(t+T-s)}f_2(x_s^{(n-1)},y_s^{(n-1)})ds \\
&+ \int_{-T}^te^{-U(t+T-s)}g_2(x_s^{(n-1)},y_s^{(n-1)})dW_s,
\end{array}
\right.
\end{align}
where $t\in  [-T,0]$. The conditional expectation given $\F_t$ is introduced to guarantee that the solution is adapted, i.e., measurable  with respect to the filtration  $\F_t$, as is done in the theory of backward \textsc{sde}s \cite{Pardous, Peng}.

Our goal is to find $\Phi_T$ via looking for $y(0,\om)$ for given
$x_0$ since $y(0) := \Phi_T(x_0,\om)$. We now introduce a time
discretization of the above iteration. Note that for
the backward part $x$, the conditional expectation is still
involved; for the forward part $y$, we use the Euler-Maruyama scheme in the
discretization. In the numerical computation, every simulation
corresponds to one sample, however, we cannot specify which sample
we have actually chosen, but we do know that it is a validated
sample in the sample space $\Om$.

Suppose $h:=T/N,\quad t_i = -T+ih,\quad i=0, 1,\ldots, N$, taking
$W_t$ to be one dimensional Wiener process, and denoting $\DD
W_i:=W_{i+1}-W_i$, then a time discretization of \eqref{eq4} is

\begin{align}\label{eq8}
\left\{
\begin{array}{rl}
x_{t_i}^{(n)}&=\EX\bigg\{x_0+\sum_{j=i}^{N-1}\bigg[Sx_{t_j}^{(n-1)}
-f_1(x_{t_j}^{(n-1)},y_{t_j}^{(n-1)})\bigg]h\,\vline\,\F_{t_i}\bigg\},\\
y_{t_i}^{(n)}&=e^{-U(t_i+T)}y(-T)+\sum_{j=0}^{i-1}e^{-U(t_i-t_j)}f_2(x_{t_j}^{(n-1)},y_{t_j}^{(n-1)})h\\
&+\sum_{j=0}^{i-1}e^{-U(t_i-t_j)}g_2(x_{t_j}^{(n-1)},y_{t_j}^{(n-1)})\DD
W_j.
\end{array}
\right.
\end{align}
In Section~\ref{sec4}, we will explain how to implement this scheme.
Note that we could also use the  available $(x_{t_j}^{(n)}, y_{t_j}^{(n)})$, instead of $(x_{t_j}^{(n-1)}, y_{t_j}^{(n-1)})$, at the current iteration (for $j=1, \cdots, i-1$) to calculate $y_{t_i}^{(n)}$ in the above scheme~\eqref{eq8}.

We now show the convergence of the above time discretized Picard
iteration scheme \eqref{eq8}.

\subsection{Convergence of the numerical scheme}

We now prove the convergence of the numerical scheme \eqref{eq8} devised in the last subsection. Namely, we prove the convergence of $(x_t^{(n)}, y_t^{(n)})$ to $(x_t, y_t)$ in a certain sense (see Theorem~\ref{thm1} below). To this end, we estimate $x_t^{(n)}-x_t^{(\infty)}, y_t^{(n)}-y_t^{(\infty)}$ and $x_t^{(\infty)}-x_t, y_t^{(\infty)}-y_t$ in Lemma~\ref{L2} and Lemma~\ref{L3}, respectively.

Define

\begin{align*}
\left\{
\begin{array}{rl}
x_{t_N}^{(\infty)}&=x_0\\
y_{t_0}^{(\infty)}&=0\\
x_{t_i}^{(\infty)}&=\EX\{\bigg(x_0+\sum_{j=i}^{N-1}[Sx_{t_{j}}^{(\infty)}-f_1(x_{t_{j}}^{(\infty)},y_{t_{j}}^{(\infty)})]h\bigg)\mid\F_{t_i}\}\\
%x_{t_i}^{(\infty)}&=\EX\{x_{t_{i+1}}^{(\infty)}-[Sx_{t_{i+1}}^{(\infty)}+f_1(x_{t_{i+1}}^{(\infty)},y_{t_{i+1}}^{(\infty)})]h\mid\F_{t_i}\}\\
y_{t_i}^{(\infty)}&=e^{-Uh}[y_{t_{i-1}}^{(\infty)}+\int_{t_{i-1}}^{t_i}e^{U(s-t_{i-1})}f_2(x_{t_{i-1}}^{(\infty)},y_{t_{i-1}}^{(\infty)})ds+\int_{t_{i-1}}^{t_i}e^{U(s-t_{i-1})}g_2(x_{t_{i-1}}^{(\infty)},y_{t_{i-1}}^{(\infty)})dW_s].
\end{array}
\right.
\end{align*}
Recall
\begin{eqnarray*}
\left\{
\begin{array}{rl}
x_{t_i}^{(0)} &= 0, \quad0\leq i\leq N-1,\quad x_{t_N}=x_0 \\
y_{t_j}^{(0)} &= 0, \quad0\leq j\leq N \\
x_{t_i}^{(n+1)}&=\EX\{\bigg(x_0+\sum_{j=i}^{N-1}[Sx_{t_{j}}^{(n)}-f_1(x_{t_{j}}^{(n)},y_{t_{j}}^{(n)})]h\bigg)\mid\F_{t_i}\}\\
y_{t_i}^{(n+1)}&=e^{-Uh}[y_{t_{i-1}}^{(n+1)}+\int_{t_{i-1}}^{t_i}e^{U(s-t_{i-1})}f_2(x_{t_{i-1}}^{(n)},y_{t_{i-1}}^{(n)})ds+\int_{t_{i-1}}^{t_i}e^{U(s-t_{i-1})}g_2(x_{t_{i-1}}^{(n)},y_{t_{i-1}}^{(n)})dW_s].
\end{array}
\right.
\end{eqnarray*}

\begin{lemma}[Iteration error]\label{L2}
Let the Lipschitz condition \eqref{eq7} be satisfied. ~Assume that
\[
\mid f_1(0,0)\mid+\mid f_2(0,0)\mid+\mid
g_2(0,0)\mid+\mid x_0\mid\leq K,
\]
and
\[
M_1:=\max\{2L_f^2h,6e^{-2Uh}(L_f^2h+L_g^2)\}\quad with\quad
\sqrt{M_1T}<1.
\]
Then the following inequality holds:
For all $0\leq i\leq N$,
\begin{align*}
&\max_{0\leq i\leq N}\bigg[\EX\mid x_{t_i}^{(n)}-x_{t_i}^{(\infty)}\mid^2+\EX\mid y_{t_i}^{(n)}-y_{t_i}^{(\infty)}\mid^2\bigg]\\
&\leq C_2(\frac{1}{2}+C_1h)^n,
\end{align*}
where $C_1$ depends on $L_f, L_g, U, T$ and   $C_2$ depends on
$L_f, L_g, K, U, T$.

\end{lemma}

\begin{proof}
First note that
\begin{align*}
x_{t_i}^{(n+1)}&=\EX\{x_{t_{i+1}}^{(n+1)}+[Sx_{t_{i}}^{(n)}-f_1(x_{t_{i}}^{(n)},y_{t_{i}}^{(n)})]h\mid\F_{t_i}\},\\
x_{t_i}^{(n)}&=\EX\{x_{t_{i+1}}^{(n)}+[Sx_{t_{i}}^{(n-1)}-f_1(x_{t_{i}}^{(n-1)},y_{t_{i}}^{(n-1)})]h\mid\F_{t_i}\}.\\
\end{align*}
We now estimate
\begin{align*}
&\EX[\mid x_{t_i}^{(n+1)}-x_{t_i}^{(n)}\mid^2]\\
&=\EX\bigg(\mid\EX\bigg\{x_{t_{i+1}}^{(n+1)}-x_{t_{i+1}}^{(n)}+\big[Sx_{t_{i}}^{(n)}-Sx_{t_{i}}^{(n-1)}
-(f_1(x_{t_{i}}^{(n)},y_{t_{i}}^{(n)})-f_1(x_{t_{i}}^{(n-1)},y_{t_{i}}^{(n-1)}))\big]h\mid\F_{t_i}\bigg\}\mid^2\bigg)\\
&\leq\EX\bigg(\EX\bigg\{\mid
x_{t_{i+1}}^{(n+1)}-x_{t_{i+1}}^{(n)}+\big[Sx_{t_{i}}^{(n)}-Sx_{t_{i}}^{(n-1)}
-(f_1(x_{t_{i}}^{(n)},y_{t_{i}}^{(n)})-f_1(x_{t_{i}}^{(n-1)},y_{t_{i}}^{(n-1)}))\big]h\mid^2\mid\F_{t_i}\bigg\}\bigg)\\
\quad&\text{(Jensen's inequality)}\\
&=\EX\bigg\{\mid
x_{t_{i+1}}^{(n+1)}-x_{t_{i+1}}^{(n)}+\big[Sx_{t_{i}}^{(n)}-Sx_{t_{i}}^{(n-1)}
-(f_1(x_{t_{i}}^{(n)},y_{t_{i}}^{(n)})-f_1(x_{t_{i}}^{(n-1)},y_{t_{i}}^{(n-1)}))\big]h\mid^2\bigg\}\\
&=\EX\bigg\{\mid
x_{t_{i+1}}^{(n+1)}-x_{t_{i+1}}^{(n)}\mid^2+\mid\big[Sx_{t_{i}}^{(n)}-Sx_{t_{i}}^{(n-1)}
-(f_1(x_{t_{i}}^{(n)},y_{t_{i}}^{(n)})-f_1(x_{t_{i}}^{(n-1)},y_{t_{i}}^{(n-1)}))\big]h\mid^2\\
&+2\mid(x_{t_{i+1}}^{(n+1)}-x_{t_{i+1}}^{(n)})\sqrt
h\mid\mid\big[Sx_{t_{i}}^{(n)}-Sx_{t_{i}}^{(n-1)}
-(f_1(x_{t_{i}}^{(n)},y_{t_{i}}^{(n)})-f_1(x_{t_{i}}^{(n-1)},y_{t_{i}}^{(n-1)}))\big]\sqrt h\mid\bigg\}\\
&\leq\EX\bigg\{\mid
x_{t_{i+1}}^{(n+1)}-x_{t_{i+1}}^{(n)}\mid^2+\mid\big[Sx_{t_{i}}^{(n)}-Sx_{t_{i}}^{(n-1)}
-(f_1(x_{t_{i}}^{(n)},y_{t_{i}}^{(n)})-f_1(x_{t_{i}}^{(n-1)},y_{t_{i}}^{(n-1)}))\big]h\mid^2\bigg\}\\
&+2\sqrt{\EX[\mid
x_{t_{i+1}}^{(n+1)}-x_{t_{i+1}}^{(n)}\mid^2h]}\sqrt{\EX\mid\big[Sx_{t_{i}}^{(n)}-Sx_{t_{i}}^{(n-1)}
-(f_1(x_{t_{i}}^{(n)},y_{t_{i}}^{(n)})-f_1(x_{t_{i}}^{(n-1)},y_{t_{i}}^{(n-1)}))\big]^2h\mid}\\
&\leq\EX\bigg\{\mid
x_{t_{i+1}}^{(n+1)}-x_{t_{i+1}}^{(n)}\mid^2+\mid\big[Sx_{t_{i}}^{(n)}-Sx_{t_{i}}^{(n-1)}
-(f_1(x_{t_{i}}^{(n)},y_{t_{i}}^{(n)})-f_1(x_{t_{i}}^{(n-1)},y_{t_{i}}^{(n-1)}))\big]h\mid^2\bigg\}\\
&+\Ga h\EX\mid
x_{t_{i+1}}^{(n+1)}-x_{t_{i+1}}^{(n)}\mid^2+\Ga^{-1}h\EX\mid
Sx_{t_{i}}^{(n)}-Sx_{t_{i}}^{(n-1)}
-(f_1(x_{t_{i}}^{(n)},y_{t_{i}}^{(n)})-f_1(x_{t_{i}}^{(n)},y_{t_{i}}^{(n-1)}))\mid^2\\
\quad&(\text{Young's inequality } ab\leq \frac{a^2}{2\Ga}+\frac{b^2\Ga}{2}\quad \text {where $\Ga>0$ will be choosen later})\\
&\leq(1+\Ga h)\bigg(\EX\mid
x_{t_{i+1}}^{(n+1)}-x_{t_{i+1}}^{(n)}\mid^2\\
&+\frac{h+\Ga^{-1}}{1+\Ga h}L_f^2\EX\mid
y_{t_{i}}^{(n)}-y_{t_{i}}^{(n-1)}\mid^2h+\frac{h+\Ga^{-1}}{1+\Ga
h}L_f^2\EX\mid x_{t_{i}}^{(n)}-x_{t_{i}}^{(n-1)}\mid^2h\bigg).
\end{align*}
Since $x_{t_N}^{(n+1)} = x_{t_N}^{(n)} = x_0 $, by iterating the last inequality, we obtain

\begin{align}\label{eq11}
&\EX[\mid x_{t_i}^{(n+1)}-x_{t_i}^{(n)}\mid^2]\nonumber\\
&\leq L_f^2(h+\Ga^{-1})\big[\sum_{j=i}^{N-1}\EX\mid
y_{t_j}^{(n)}-y_{t_j}^{(n-1)}\mid^2h+\sum_{j=i}^{N-1}\EX\mid
x_{t_j}^{(n)}-x_{t_j}^{(n-1)}\mid^2h\big].
\end{align}

For the forward \textsc{sde} part, since in the numerical scheme $y$ is only of finite dimension, we consider $y$ to be one-dimensional for simplicity, and estimate

\begin{align}\label{eq12}
&\EX\mid y_{t_i}^{(n+1)}-y_{t_i}^{(n)}\mid^2\nonumber\\
=\EX\mid &e^{-Uh}\bigg[(y_{t_{i-1}}^{(n+1)}-y_{t_{i-1}}^{(n)})+\int_{t_{i-1}}^{t_i}e^{U(s-t_{i-1})}(f_2(x_{t_{i-1}}^{(n)},y_{t_{i-1}}^{(n)})-f_2(x_{t_{i-1}}^{(n-1)},y_{t_{i-1}}^{(n-1)}))ds\nonumber\\
&+\int_{t_{i-1}}^{t_i}e^{U(s-t_{i-1})}(g_2(x_{t_{i-1}}^{(n)},y_{t_{i-1}}^{(n)})-g_2(x_{t_{i-1}}^{(n-1)},y_{t_{i-1}}^{(n-1)}))dW_s\bigg]\mid^2\nonumber\\
\leq 3e^{-2Uh}\bigg(&\EX\mid y_{t_{i-1}}^{(n+1)}-y_{t_{i-1}}^{(n)}\mid^2+\EX\mid\int_{t_{i-1}}^{t_i}e^{U(s-t_{i-1})}(f_2(x_{t_{i-1}}^{(n)},y_{t_{i-1}}^{(n)})-f_2(x_{t_{i-1}}^{(n-1)},y_{t_{i-1}}^{(n-1)}))ds\mid^2\nonumber\\
&+\EX\mid\int_{t_{i-1}}^{t_i} e^{U(s-t_{i-1})}(g_2(x_{t_{i-1}}^{(n)},y_{t_{i-1}}^{(n)})-g_2(x_{t_{i-1}}^{(n-1)},y_{t_{i-1}}^{(n-1)}))dW_s\mid^2\bigg)\nonumber\\
\leq 3e^{-2Uh}\bigg(&\EX\mid y_{t_{i-1}}^{(n+1)}-y_{t_{i-1}}^{(n)}\mid^2+h^2\EX L_f^2(\mid x_{t_{i-1}}^{(n)}-x_{t_{i-1}}^{(n-1)}\mid+\mid y_{t_{i-1}}^{(n)}-y_{t_{i-1}}^{(n-1)}\mid)^2\nonumber\\
&+h\EX L_g^2(\mid x_{t_{i-1}}^{(n)}-x_{t_{i-1}}^{(n-1)}\mid+\mid y_{t_{i-1}}^{(n)}-y_{t_{i-1}}^{(n-1)}\mid)^2\bigg)\nonumber\\
\leq 3e^{-2Uh}\bigg(&\EX\mid y_{t_{i-1}}^{(n+1)}-y_{t_{i-1}}^{(n)}\mid^2+(2h^2L_f^2+2hL_g^2)\EX\mid x_{t_{i-1}}^{(n)}-x_{t_{i-1}}^{(n-1)}\mid^2\nonumber\\
&+(2h^2L_f^2+2hL_g^2)\EX\mid y_{t_{i-1}}^{(n)}-y_{t_{i-1}}^{(n-1)}\mid^2\bigg)\nonumber\\
\leq 6e^{-2Uh}(&hL_f^2+L_g^2)\bigg(\sum_{j=1}^{i-1}\EX\mid
x_{t_j}^{(n)}-x_{t_j}^{(n-1)}\mid^2h+\sum_{j=1}^{i-1}\EX\mid
y_{t_j}^{(n)}-y_{t_j}^{(n-1)}\mid^2h\bigg),
\end{align}
where the last inequality is by iteration.

Combining \eqref{eq11} and \eqref{eq12}, and letting
$\widetilde {M_1}:=\max\{L_f^2(h+\Ga^{-1}),6e^{-2Uh}(L_f^2h+L_g^2)\}$, we have
\begin{align*}
&\max_{0\leq i\leq N}\bigg[\EX\mid x_{t_i}^{(n+1)}-x_{t_i}^{(n)}\mid^2+\EX\mid y_{t_i}^{(n+1)}-y_{t_i}^{(n)}\mid^2\bigg]\\
&\leq \widetilde {M_1}T\max_{0\leq i\leq N}\bigg[\EX\mid x_{t_i}^{(n)}-x_{t_i}^{(n-1)}\mid^2+\EX\mid y_{t_i}^{(n)}-y_{t_i}^{(n-1)}\mid^2\bigg]\\
&\leq (\widetilde {M_1}T)^n\max_{0\leq i\leq N}\bigg[\EX\mid x_{t_i}^{(1)}-x_{t_i}^{(0)}\mid^2+\EX\mid y_{t_i}^{(1)}-y_{t_i}^{(0)}\mid^2\bigg]\\
&=(\widetilde{M_1}T)^n\max_{0\leq i\leq N}\bigg[\EX\mid x_{t_i}^{(1)}\mid^2+\EX\mid y_{t_i}^{(1)}\mid^2\bigg].
\end{align*}

Define $\bigg((\widetilde{M_1}T)^n\max_{0\leq i\leq N}\bigg[\EX\mid
x_{t_i}^{(1)}\mid^2+\EX\mid
y_{t_i}^{(1)}\mid^2\bigg]\bigg)^{\frac{1}{2}}$ by $a_n$.  Then the
series $\sum_na_n$ converges when $\sqrt{\widetilde M_1T}<1$. Therefore, for
all $0\leq i\leq N$, $\{(x_{t_i}^{(n)}, y_{t_i}^{(n)})\}_n$ is a
Cauchy sequence and thus converges to
$(x_{t_i}^{(\infty)},y_{t_i}^{(\infty)})$ in the mean square sense.
To be more precise,
\begin{align*}
&\max_{0\leq i\leq N}\bigg[\EX\mid x_{t_i}^{(\infty)}-x_{t_i}^{(n)}\mid^2+\EX\mid y_{t_i}^{(\infty)}-y_{t_i}^{(n)}\mid^2\bigg]\\
&\leq \bigg(\sum_{\nu=n}^{\infty}a_\nu\bigg)^2\\
&=\frac{(\widetilde M_1T)^n\max_{0\leq i\leq N}\bigg[\EX\mid
x_{t_i}^{(1)}\mid^2+\EX\mid
y_{t_i}^{(1)}\mid^2\bigg]}{(1-\sqrt{\widetilde M_1T})^2}.
\end{align*}

Choosing $\Ga:=1/h$, $\widetilde M_1 = M_1$ as in the assumption, thus $\sqrt{M_1T}<1$. By taking
$M_1T:=\frac{1}{2}+C_1h$ where $C_1$ depends on $L_f, L_g, U, T$ and
$C_1h<\frac{1}{2}$, since $\max_{0\leq i\leq N}\EX\mid x_{t_i}^{(1)}\mid^2+\EX\mid
y_{t_i}^{(1)}\mid^2\leq C_3$ as shown similarly as Lemma 8 in Bender and Denk~
\cite{Bender}   and $C_3$ only depends on $L_f, L_g, K$, we have
\begin{align*}
&\max_{0\leq i\leq N}\bigg[\EX\mid x_{t_i}^{(\infty)}-x_{t_i}^{(n)}\mid^2+\EX\mid y_{t_i}^{(\infty)}-y_{t_i}^{(n)}\mid^2\bigg]\\
&\leq C_2(\frac{1}{2}+C_1h)^n,
\end{align*}
where $C_2$ depends on $L_f, L_g, K, U, T$. This proves the lemma.

\end{proof}

In the following, $C$ will denote a generic positive constant, independent of $i$ and $n$, that
may take different values from line to line.

%In order to prove Theorem 1, we need to introduce an intermediate
%part of the solution $(x^{(\infty)}, y^{(\infty)})$, which are
%defined as
%\begin{align*}
%\left\{
%\begin{array}{rl}
%x_{t_N}^{(\infty)}&=x_0\\
%x_{t_i}^{(\infty)}&=\EX\{x_{t_{i+1}}^{(\infty)}-[Sx_{t_{i+1}}^{(\infty)}+f_1(x_{t_{i+1}}^{(\infty)},y_{t_{i+1}}^{(\infty)})]h\mid\F_{t_i}\}\\
%y_{t_0}^{(\infty)}&=0\\
%y_{t_i}^{(\infty)}&=e^{-Uh}[y_{t_{i-1}}^{(\infty)}+f_2(x_{t_{i-1}}^{(\infty)},y_{t_{i-1}}^{(\infty)})h+g_2(x_{t_{i-1}}^{(\infty)},y_{t_{i-1}}^{(\infty)})\DD
%W_{i-1}]
%\end{array}
%\right.
%\end{align*}
%We will prove Theorem 1 by showing the convergence of
%$(x^{(n)},y^{(n)})$ to $(x^{(\infty)},y^{(\infty)})$ and the
%convergence of $(x^{(\infty)},y^{(\infty)})$ to $(x,y)$, which
%will be proved in Lemma~\ref{L2} and Lemma~\ref{L3}, respectively.

\begin{lemma}[Discretization error]\label{L3}
Assume all the conditions as in Lemma \ref{L2} are satisfied,
%Define for $t\in[t_{i-1},t_i]$
%\begin{align*}
%x_t&=\EX\bigg[x_{t_i}+\int_{t}^{t_i}(Sx_s-f_1(x_s,y_s))ds\vline\F_t\bigg],\\
%x_t^{(\infty)}&=\EX\bigg[x_{t_i}^{(\infty)}+\int_{t}^{t_i}(Sx_{t_{i-1}}^{(\infty)}-f_1(x_{t_{i-1}}^{(\infty)},y_{t_{i-1}}^{(\infty)}))ds\vline\F_t\bigg],\\
%y_t&=e^{-Uh}\bigg(y_{t_{i-1}}+\int_{t_{i-1}}^tf_2(x_s,y_s)ds+\int_{t_{i-1}}^tg_2(x_s,y_s)dW_s\bigg),\\
%y_t^{(\infty)}&=e^{-Uh}\bigg(y_{t_{i-1}}^{(\infty)}+\int_{t_{i-1}}^tf_2(x_{t_{i-1}}^{(\infty)},y_{t_{i-1}}^{(\infty)})ds+\int_{t_{i-1}}^tg_2(x_{t_{i-1}}^{(\infty)},y_{t_{i-1}}^{(\infty)})dW_s\bigg).\\
%\end{align*}
then
\begin{align*}
\sup_{-T\leq t\leq0}\bigg(\EX\mid x_t-x_t^{(\infty)}\mid^2+\EX\mid y_t-y_t^{(\infty)}\mid^2\bigg)<Ch.
\end{align*}
\end{lemma}

\begin{proof}
Note that
\begin{align*}
x_{t_{i-1}}^{(\infty)} = \EX[x_{t_{i}}^{(\infty)}\mid\F_{t_{i-1}}]
+[Sx_{t_{i-1}}^{(\infty)}-f_1(x_{t_{i-1}}^{(\infty)},y_{t_{i-1}}^{(\infty)})](t_i-t_{i-1}).
\end{align*}
For $t_{i-1}\leq t\leq t_i$, we have
\begin{align*}
x_{t_{i-1}}^{(\infty)} = \EX[x_{t}^{(\infty)}\mid\F_{t_{i-1}}]
+[Sx_{t_{i-1}}^{(\infty)}-f_1(x_{t_{i-1}}^{(\infty)},y_{t_{i-1}}^{(\infty)})](t-t_{i-1}).
\end{align*}
By the Martingale Representation Theorem \cite{Klebaner},
\[
x_t^{(\infty)} = \EX[x_t^{(\infty)}\mid\F_{t_{i-1}}]+\int_{t_{i-1}}^{t}Z_s^{(\infty)}dW_s,
\]
we have
\begin{align*}
x_{t}^{(\infty)} = x_{t_{i-1}}^{(\infty)}-(t-t_{i-1})[Sx_{t_{i-1}}^{(\infty)}-f_1(x_{t_{i-1}}^{(\infty)},y_{t_{i-1}}^{(\infty)})]
+\int_{t_{i-1}}^{t}Z_s^{(\infty)}dW_s,
\end{align*}
which yields
\[
dx_t^{(\infty)} = [-Sx_{t_{i-1}}^{(\infty)} + f_1(x_{t_{i-1}}^{(\infty)},y_{t_{i-1}}^{(\infty)})]dt + Z_t^{(\infty)}dW_t.
\]
By It\^{o}'s formula, we have
\begin{align*}
\EX[x_t-x_t^{(\infty)}]^2 &= \EX[x_{t_i}-x_{t_i}^{(\infty)}]^2\\
&-\EX\int_t^{t_i}2(x_s-x_s^{(\infty)})[-S(x_s-x_{t_{i-1}}^{(\infty)})+f_1(x_s,y_s)-f_1(x_{t_{i-1}}^{(\infty)},y_{t_{i-1}}^{(\infty)})]ds\\
&-\EX\int_t^{t_i}(Z_s-Z_s^{(\infty)})^2ds.
\end{align*}
Let
\begin{align*}
&\delta x_t:=x_t-x_t^{(\infty)},\quad \delta Z_t:=Z_t-Z_t^{(\infty)},\\
\text{ and } &\delta f_t:=(Sx_t-f_1(x_t,y_t))-(Sx_{t_{i-1}}^{(\infty)}-f_1(x_{t_{i-1}}^{(\infty)},y_{t_{i-1}}^{(\infty)})),
\end{align*}
then define
\begin{align*}
A_t&:=\EX\mid\delta x_t\mid^2 + \int_t^{t_i}\EX\mid\delta Z_s\mid^2ds-\EX\mid\delta x_{t_i}\mid^2\\
&\int_t^{t_i}\EX[2\delta x_s\delta f_s]ds\\
&\leq \EX[C\int_t^{t_i}\mid\delta x_s\mid(\mid x_s-x_{t_{i-1}}^{(\infty)}\mid+\mid y_s-y_{t_{i-1}}^{(\infty)}\mid)]\\
&\leq\int_t^{t_i}\a\EX\mid\delta x_s\mid^2ds + \frac{C}{\a}\int_{t}^{t_i}\EX[\mid x_s-x_{t_{i-1}}^{(\infty)}\mid^2+\mid y_s-y_{t_{i-1}}^{(\infty)}\mid^2]ds.
\end{align*}
Since
\begin{align*}
\EX\mid x_s-x_{t_{i-1}}\mid^2 &= \EX\mid\int_{t_{i-1}}^s(-Sx_r+f_1(x_r,y_r))dr+\int_{t_{i-1}}^sZ_rdW_r\mid^2\\
&\leq2\big[\EX\mid\int_{t_{i-1}}^s(-Sx_r+f_1(x_r,y_r))dr\mid^2+\EX\mid\int_{t_{i-1}}^sZ_rdW_r\mid^2\big]\\
&\leq2\big[(s-t_{i-1})\EX\int_{t_{i-1}}^s(-Sx_r+f_1(x_r,y_r))^2dr+\EX\int_{t_{i-1}}^sZ_r^2dr\big]\\
&\leq2\big[(s-t_{i-1})\int_{t_{i-1}}^s[\EX(-Sx_r+f_1(x_r,y_r)-f_1(0,0))^2+\EX(f_1(0,0))^2]dr\\
&+\EX\int_{t_{i-1}}^sZ_r^2dr\big]\\
&\leq2[(s-t_{i-1})\int_{t_{i-1}}^s\EX(Cx_r^2+2L_f(x_r^2+y_r^2)+C)dr+\EX\int_{t_{i-1}}^sZ_r^2dr]\\
&\leq Ch,
\end{align*}
then we have
\begin{align*}
\EX\mid x_s-x_{t_{i-1}}^{(\infty)}\mid^2&\leq2[\EX\mid x_s-x_{t_{i-1}}\mid^2+\EX\mid\delta x_{t_{i-1}}\mid^2]\\
&\leq C(h+\EX\mid\delta x_{t_{i-1}}\mid^2)
\end{align*}
for $t_{i-1}\leq t\leq s<t_i$.
On the other hand,
\begin{align*}
\EX\mid y_s-y_{t_{i-1}}\mid^2 &\leq C
+ C\EX\mid\int_{t_{i-1}}^se^{U(r+T)}f_2(x_r,y_r)dr+\int_{t_{i-1}}^se^{U(r+T)}g_2(x_r,y_r)dW_r\mid^2\\
&\leq C+C[(s-t_{i-1})\EX\int_{t_{i-1}}^se^{2U(r+T)}f_2^2(x_r,y_r)dr+\EX\int_{t_{i-1}}^se^{2U(r+T)}g_2^2(x_r,y_r)dr]\\
&\leq Ch,
\end{align*}
and note also that $y^{(\infty)}$ agrees with $y$ at each grid point, i.e. $y_{t_{i-1}}^{(\infty)}=y_{t_{i-1}}$, thus
\begin{align*}
\EX\mid y_s-y_{t_{i-1}}^{(\infty)}\mid^2&\leq2[\EX\mid y_s-y_{t_{i-1}}\mid^2+\EX\mid\delta y_{t_{i-1}}\mid^2]\\
&\leq C(h+\EX\mid\delta y_{t_{i-1}}\mid^2)\\
&=Ch.
\end{align*}
By the definition of $A_t$,
\begin{align*}
A_t&\leq\int_{t}^{t_i}\a\EX\mid\delta x_s\mid^2ds + \frac{C}{\a}\int_{t}^{t_i}[h+\EX\mid\delta x_{t_{i-1}}\mid^2]ds\\
&\leq \int_{t}^{t_i}\a\EX\mid\delta x_s\mid^2ds + \frac{C}{\a}[h^2+h\EX\mid\delta x_{t_{i-1}}\mid^2]
\end{align*}
and
\begin{align}\label{eq:delta x}
\EX\mid\delta x_t\mid^2&\leq\EX\mid\delta x_t\mid^2 + \int_t^{t_i}\EX\mid\delta Z_s\mid^2ds\nonumber\\
&\leq\int_t^{t_i}\a\EX\mid\delta x_{s}\mid^2ds + \bigg(\EX\mid\delta x_{t_{i}}\mid^2+\frac{C}{\a}[h^2+h\EX\mid\delta x_{t_{i-1}}\mid^2]\bigg).
\end{align}
Let $B_i:=\EX\mid\delta x_{t_{i}}\mid^2+\frac{C}{\a}[h^2+h\EX\mid\delta x_{t_{i-1}}\mid^2]$, then by Grownwall's inequality,
$\EX\mid\delta x_t\mid^2\leq B_ie^{\int_t^{t_i}\a ds}\leq B_ie^{\a h}$. Plugging it in the second inequality of \eqref{eq:delta x}, we have
\begin{align}\label{eq:Bi}
\EX\mid\delta x_t\mid^2 + \int_t^{t_i}\EX\mid\delta Z_s\mid^2ds\leq B_i(1+\a e^{\a h}h).
\end{align}
By taking $t = t_{i-1}$, we have
\begin{align*}
\EX\mid\delta x_{t_{i-1}}\mid^2 + \int_{t_{i-1}}^{t_i}\EX\mid\delta Z_s\mid^2ds\leq(1+Ch)(\EX\mid\delta x_{t_{i}}\mid^2+\frac{C}{\a}[h^2+h\EX\mid\delta x_{t_{i-1}}\mid^2]),
\end{align*}
and if we choose $\a\gg C$,
\begin{align}\label{eq:delta xi}
\EX\mid\delta x_{t_{i-1}}\mid^2 + \int_{t_{i-1}}^{t_i}\EX\mid\delta Z_s\mid^2ds\leq(1+Ch)(\EX\mid\delta x_{t_{i}}\mid^2+h^2).
\end{align}
Iterating the last inequality and recall that $\EX\mid\delta x_{t_N}\mid^2 = 0$, we have for sufficiently small $h$,
\begin{align*}
\EX\mid\delta x_{t_{i-1}}\mid^2\leq(1+Ch)^N(\EX\mid\delta x_{t_N}\mid^2+h)\leq Ch,
\end{align*}
which yields $B_i\leq Ch$, and by \eqref{eq:Bi}, we get for all $t\in[t_{i-1},t_i]$
\begin{align*}
\EX\mid\delta x_t\mid^2\leq Ch,
\end{align*}
and the right hand side of the inequality does not depend on $t$. Therefore,
\begin{align}\label{eq:delta x_t}
\sup_{-T\leq t\leq0}\EX\mid\delta x_t\mid^2 = \sup_{-T\leq t\leq0}\EX\mid x_t^{(\infty)}-x_t\mid^2\leq Ch.
\end{align}
For the forward part $y$, we have
\begin{align}\label{eq:delta y_t}
\EX\mid y_t^{(\infty)}-y_t\mid^2 &=\EX\bigg(\int_{t_{i-1}}^te^{-U(t-s)}[f_2(x_{t_{i-1}}^{(\infty)},y_{t_{i-1}}^{(\infty)})-f_2(x_s,y_s)]ds \nonumber\\
&+\int_{t_{i-1}}^te^{-U(t-s)}[g_2(x_{t_{i-1}}^{(\infty)},y_{t_{i-1}}^{(\infty)})-g_2(x_s,y_s)]dW_s\bigg)^2\nonumber\\
&\leq 2 h\EX\int_{t_{i-1}}^te^{-2U(t-s)}[f_2(x_{t_{i-1}}^{(\infty)},y_{t_{i-1}}^{(\infty)})-f_2(x_s,y_s)]^2ds\nonumber\\
&+2\EX\int_{t_{i-1}}^te^{-2U(t-s)}[g_2(x_{t_{i-1}}^{(\infty)},y_{t_{i-1}}^{(\infty)})-g_2(x_s,y_s)]^2ds\nonumber\\
&\leq[C(h+\EX\mid\delta x_{t_{i-1}}\mid^2)+C(h+\EX\mid\delta x_{t_{i-1}}\mid^2)](2h^2C+2hC)\nonumber\\
&\leq Ch^2,
\end{align}
where the last inequality is due to \eqref{eq:delta xi}.
Consider \eqref{eq:delta x_t} and \eqref{eq:delta y_t}, we have
 \[\sup_{-T\leq t\leq0}\EX\mid x_t^{(\infty)}-x_t\mid^2+\EX\mid y_t^{(\infty)}-y_t\mid^2 \leq Ch.\]
This proves the lemma.
\end{proof}

Now we are ready to state the convergence   theorem.
For $t\in[t_{i-1}, t_i]$, define $x_t^{(n)} = x_{t_{i-1}}^{(n)}$ and $y_t^{(n)} = y_{t_{i-1}}^{(n)}$. We can also define $x_t^{(n)}$ and $y_t^{(n)}$ as the linear interpolation among $x_{t_{i-1}}^{(n)}$'s and among $y_{t_{i-1}}^{(n)}$'s, respectively, and the following result also holds.

\begin{theorem}[Convergence]\label{thm1} Assume that all the conditions in Lemma~\ref{L2} are satisfied.
%\begin{eqnarray*}
%x_{t_i}^{(0)} &\equiv& 0, \quad0\leq i\leq N-1,\quad x_{t_N}=x_0, \\
%y_{t_j}^{(0)} &\equiv& 0, \quad0\leq j\leq N, \\
%x_{t_i}^{(n+1)}&=&\EX\{x_{t_{i+1}}^{(n+1)}-[Sx_{t_{i}}^{(n)}+f_1(x_{t_{i}}^{(n)},y_{t_{i}}^{(n)})]h\mid\F_{t_i}\},\\
%y_{t_i}^{(n+1)}&=&e^{-Uh}[y_{t_{i-1}}^{(n+1)}+\int_{t_{i-1}}^{t_i}f_2(x_{t_{i-1}}^{(n)},y_{t_{i-1}}^{(n)})ds+\int_{t_{i-1}}^{t_i}g_2(x_{t_{i-1}}^{(n)},y_{t_{i-1}}^{(n)})dW_s].
%\end{eqnarray*}
%Moreover, assume that
%\[
%\sup_{-T\leq t\leq0}(\mid f_1(0,0)\mid+\mid f_2(0,0)\mid+\mid
%g_2(0,0)\mid+\mid x_0\mid)\leq K,
%\]
%and
%\[
%M_1:=\max\{2L_f^2h,6e^{-2Uh}(L_f^2h+L_g^2)\}\quad with\quad
%\sqrt{M_1T}<1.
%\]
Then
\begin{align*}
\sup_{-T\leq t\leq0}\bigg(\EX\mid x_t-x_t^{(n)}\mid^2+\EX\mid y_t-y_t^{(n)}\mid^2\bigg)\leq
C(h+(\frac{1}{2}+Ch)^n),
\end{align*}
where $C>0$ will denote a generic constant depending only on the data $L_f, L_g, K, U, T$.
This implies the convergence of the numerical scheme \eqref{eq8}, as $h\to0$ and $n\to\infty$.
\end{theorem}

\begin{proof}
Note that
\begin{align*}
&\EX\mid x_t-x_t^{(n)}\mid^2+\EX\mid y_t-y_t^{(n)}\mid^2\\
\leq&\EX\mid x_t-x_{t_{i-1}}\mid^2+\EX\mid x_{t_{i-1}}-x_{t_{i-1}}^{(\infty)}\mid^2+\EX\mid x_{t_{i-1}}^{(\infty)}-x_{t_{i-1}}^{(n)}\mid^2\\
+&\EX\mid y_t-y_{t_{i-1}}\mid^2+\EX\mid y_{t_{i-1}}-y_{t_{i-1}}^{(\infty)}\mid^2+\EX\mid y_{t_{i-1}}^{(\infty)}-y_{t_{i-1}}^{(n)}\mid^2\\
\leq& C(h+(\frac{1}{2}+Ch)^n)
\end{align*}
by regularity of the true solution as well as Lemma~\ref{L2} and Lemma~\ref{L3}. Taking supremum on both sides of the above inequality, we have
\begin{align*}
\sup_{-T\leq t\leq0}\bigg(\EX\mid x_t-x_t^{(n)}\mid^2+\EX\mid y_t-y_t^{(n)}\mid^2\bigg)\leq C(h+(\frac{1}{2}+Ch)^n).
\end{align*}

\end{proof}

\bigskip

\section{Approximation of the stochastic inertial manifold}\label{sec4}

Now we approximate the graph $\Phi$ for the inertial manifold $
\mathcal{M}$. Recall that $\Phi$ is approximated via
\begin{equation}\label{mfd999}
y(0,\om):=\Phi_T(x_0,\om),\ \om\in\Om.
\end{equation}
When $T$ is sufficiently big, $\Phi_T\approx\Phi$. So we need to
evaluate $y(0, \om)$ (which is y(0) in Section 3). To this end, we
need to compute, step by step, $x_{t_i}^{(n)}, y_{t_i}^{(n)}$ in
\eqref{eq8}. For the backward part $x_{t_i}^{(n)}$, the conditional
expectations $\EX[\cdot\mid\F_{t_i}]\in L^2(\F_{t_i})$~in \eqref{eq8}
will be approximated by their orthogonal projections $P_i$ on finite
dimensional subspaces $\L_i$ of $L^2(\F_{t_i})$, where
$L^2(\F_{t_i})$ contains all the functions in $L^2(\Om)$ that are
$\F_{t_i}$-adapted.
%an orthonormal projection on finite dimensional subspaces and the
%coefficients of the orthonormal projection are calculated by the
%Monte Carlo least squares  method \cite{Bender, Jentzen}.

%We now approximate the conditional expectations
%$\EX[\cdot\mid\F_{t_i}]$ by their orthogonal projections $P_i$ on
%finite dimensional subspaces $\L_i$ of $L^2(\F_{t_i})$, where
%$L^2(\F_{t_i})$ contains all the functions in $L^2(\Om)$ that are
%$\F_{t_i}$-adapted.
Indeed, instead of computing $x_{t_i}^{(n)}$ as follows
\[
x_{t_i}^{(n)}=\EX\bigg\{\bigg[x_0-\sum_{j=i}^{N-1}(S
x_{t_j}^{(n-1)}+f_1(
x_{t_j}^{(n-1)},y_{t_j}^{(n-1)}))\DD_j\bigg]\vline\F_{t_i}\bigg\},
\]
we will compute its orthogonal projection $\hat x_{t_i}^{(n)}$:
\[
\hat x_{t_i}^{(n)}=P_i\EX\bigg\{\bigg[x_0-\sum_{j=i}^{N-1}(S\hat
x_{t_j}^{(n-1)}+f_1(\hat
x_{t_j}^{(n-1)},y_{t_j}^{(n-1)}))\DD_j\bigg]\vline\F_{t_i}\bigg\}.
\]

Denoting a basis of the projection space $\L_{i}$ by
$\{\eta_1^i,\ldots,\eta_{D(i)}^i\}$, we then have
\begin{eqnarray} \label{xxx}
\hat x_{t_i}^{(n)}=\sum_{d=1}^{D(i)}\a_{i,d}^{(n)}\eta_d^i.
\end{eqnarray}
Set $\eta^i:=(\eta_1^i,\ldots,\eta_{D(i)}^i)'$ and $
\a_i^{(n)}:=(\a_{i,1}^{(n)},\ldots,\a_{i,D(i)}^{(n)})'$, with prime
denoting matrix transpose, the calculation for the backward part
$x_{t_i}^{n}$ boils down to two aspects: one is the basis $\eta^i$,
the other is the coefficient $\a_i^n$.

\begin{remark}[Basis] \label{remarkbasis}
As in \cite{Luo}, a complete orthonormal basis of $L^2(\Om, \F_t,
\mathbb P)$ is given by the Wick polynomials
$\{T_a^t(\xi);a\in\mathcal {J}\}$,
\begin{align*}
T_a^t(\xi):=\Pi_{i=1}^\infty H_{a_i}(\xi_i(t)),
\end{align*}
where $H_{a_i}(\cdot)$ are Hermite polynomials. Here
\[
\mathcal J:=\{a=(a_i,i\geq1)\mid a_i\in\{0,1,2,\ldots\},\mid
a\mid=\sum_{i=1}^\infty a_i<\infty\},
\]

\[
\xi_i(t):=\int_0^tm_i(s)dW_s,
\]
and $ m_i(s), i = 1, 2, \ldots $, are
  a set of complete orthonormal basis in
the Hilbert space $L^2([0,t])$.  One choice of the basis elements
are the Hermite polynomials of Brownian motion $W_t$. In fact,
\begin{align*}
&a^1 = (1,0,0,\cdots),\quad T_{a^1}^t(\xi(t))=
H_1(\xi_1(t))H_0(\xi_2(t))H_0(\xi_3(t))\cdots = H_1(\xi_1(t)),\\
&a^2 = (2,0,0,\cdots),\quad T_{a^2}^t(\xi(t))=
H_2(\xi_1(t))H_0(\xi_2(t))H_0(\xi_3(t))\cdots = H_2(\xi_1(t)),\\
&\quad\vdots \\
&a^n = (n,0,0,\cdots),\quad T_{a^n}^t(\xi(t))=
H_2(\xi_1(t))H_0(\xi_2(t))H_0(\xi_3(t))\cdots = H_2(\xi_1(t)).
\end{align*}
Since $\xi_1(t)=\int_0^tm_1(s)dW_s = \int_0^t1dW_s = W_t$, the basis
elements $\eta^i_d$'s are $H_d(W_{t_i})$ which are adopted in our
numerical implementation.

\end{remark}

In principle, the coefficients $\a_{i}^{(n)}$ are calculated as
follows:
\begin{align}\label{eq5}
\a_{i}^{(n)}=\beta_i^{-1}\EX\bigg[\eta^i\bigg(x_0-\sum_{j=i}^{N-1}(S\hat
x_{t_j}^{(n-1)} +f_1(\hat
x_{t_j}^{(n-1)},y_{t_j}^{(n-1)}))\DD_j\bigg)\bigg],
\end{align}
where
$\beta_i:=\bigg(\EX[\eta_p^i\eta_q^i]\bigg)_{p,q=1,\ldots,D(i)}$ are
the inner-product matrices associated with the basis.

In practice, $\beta_i's$ are computed by their simulation-based
estimators, such as Monte Carlo least squares estimators used here
\cite{Bender}. To this end, we are assuming to have $r$
(sufficiently large) independent copies $(_\l W_i,_\l\eta_d^i)$, $\l
= 1,\ldots, r$, of $(W_i,\eta_d^i)$. Here the index $\l$ denotes
copies. Then
\begin{align*}
\beta_i = \bigg(\EX[\eta_p^i\eta_q^i]\bigg)_{p,q=1,\ldots,D(i)} = \left(  \begin{array}{cccc} \EX\eta_1^i\eta_1^i & \EX\eta_1^i\eta_2^i & \cdots & \EX\eta_1^i\eta_{D(i)}^i\\
                                      \EX\eta_2^i\eta_1^i & \EX\eta_2^i\eta_2^i & \cdots & \EX\eta_2^i\eta_{D(i)}^i\\
                                      \vdots & \vdots & \cdots & \vdots\\
                                      \EX\eta_{D(i)}^i\eta_1^i & \EX\eta_{D(i)}^i\eta_2^i & \cdots & \EX\eta_{D(i)}^i\eta_{D(i)}^i \end{array}
\right),
\end{align*}
is replaced by its Monte Carlo simulation $\bar{\beta_i}$
\begin{align*}
\bar{\beta_i}\triangleq\frac{1}{r} \left(  \begin{array}{cccc} \sum_{\l=1}^{r}{_\l\eta_1^i}{_\l\eta_1^i} & {\sum_{\l=1}^{r}}{_\l\eta_1^i}{_\l\eta_2^i} & \cdots & {\sum_{\l=1}^{r}}{_\l\eta_1^i}{_\l\eta_{r}^i}\\
                                      {\sum_{\l=1}^{r}}{_\l\eta_2^i}{_\l\eta_1^i} & {\sum_{\l=1}^{r}}{_\l\eta_2^i}{_\l\eta_2^i} & \cdots & {\sum_{\l=1}^{r}}{_\l\eta_2^i}{_\l\eta_{r}^i}\\
                                      \vdots & \vdots & \cdots & \vdots\\
                                      {\sum_{\l=1}^{r}}{_\l\eta_{r}^i}{_\l\eta_1^i} & {\sum_{\l=1}^{r}}{_\l\eta_{r}^i}{_\l\eta_2^i} & \cdots & {\sum_{\l=1}^{r}}{_\l\eta_{r}^i}{_\l\eta_{r}^i} \end{array}
\right),
\end{align*}
and is further rewritten as
\[
\bar{\beta_i}=(\bar{\A_i)}'\bar{\A_i}=\frac{1}{r}\bigg(\sum_{\l=1}^{r}{_\l\eta_a^i}{_\l\eta_b^i}\bigg),\
a,b=1,\ldots,D(i),
\]
where \begin{align*} \bar{\mathcal
{A}_i}&=\frac{1}{\sqrt{r}}\bigg({_\l\eta_d^i}\bigg)\\%\\l=1,\ldots,r,\quad
%d=1,\ldots,D(i)
&=\frac{1}{\sqrt{r}} \left(  \begin{array}{cccc}
_1\eta_1^i &
_1\eta_2^i & \cdots & _1\eta_{D(i)}^i\\
_2\eta_1^i & _2\eta_2^i & \cdots & _2\eta_{D(i)}^i\\
\vdots & \vdots & \cdots & \vdots\\
_{r}\eta_1^i & _{r}\eta_2^i & \cdots & _{r}\eta_{D(i)}^i
\end{array} \right).
\end{align*}
%and $_\l\eta_d^i, \l = 1,\ldots,D(i)$ are independent copies of the
%basis functions $\eta_d^i$.
%The expectations in \eqref{eq8} are computed by their
%simulation-based estimators.
%In the implementation, we need at least $D(i)$
%independent copies of $(_\l\eta_d^i), \l=1,\ldots,D(i)$.
%is a simulation formula for $\beta_i$.
The pseudo-inverse denoted by $A^\dagger:=(A'A)^{-1}A'$ is used in
the computation of $\bar{\beta_i}$.

The calculation of the coefficients $\a_i^{(n)}$ is to obtain the
backward part $x_{t_i}^{(n)}$. %Once $\a_i^{(n)}$ is available, $\hat
%x_{t_i}^{(n)}$ is calculated through \eqref{xxx} by fixing $\om$ in
%basis functions $\eta_d^i$,.
Note that the calculation of $\a_i^{(n)}$ also needs the forward
part $y_{t_i}^{(n-1)}$, which is calculated through Euler-Maruyama scheme. The
following formulae illustrate how to update $\a_i^{(n)}$.

%where we take r copies
%$\a_i^{(n,r)}$ are column vectors of
%copies
%of $\a_i^{(n)}, i = 0, 1, 2,\ldots, N-1$.
%For simplicity, $\hat
%x_{t_j}^{(n)}$ is one dimension in the following formulas.
\begin{align}
\a_i^{(0)}&=(\underbrace{0, 0, \ldots, 0}_{D(i)\text{ many}})'\label{eq:alpha0}, \\
(_\l\hat {x}_{t_i}^{(n-1)})_{\l=1,\ldots,r}&\triangleq(_1\hat
x_{t_i}^{(n-1)}, _2\hat
x_{t_i}^{(n-1)}, \cdots, _{r}\hat x_{t_i}^{(n-1)})'\nonumber\\
&=(\sum_{d=1}^{D(i)}\a_{i,d}^{(n-1)}{_1\eta_d^i}, \sum_{d=1}^{D(i)}\a_{i,d}^{(n-1)}{_2\eta_d^i}, \cdots, \sum_{d=1}^{D(i)}\a_{i,d}^{(n-1)}{_{r}\eta_d^i})'\label{eq:xrandom}, \\
(_\l\hat y^{(0)})_{\l=1,\ldots,r}&=\textbf{0},\label{eq:y0} \\
_\l{\textbf{f}}(t_j)^{(n-1)}&\triangleq\bigg({f(t_j, _\l\hat
x_{t_j}^{(n-1)},_\l\hat y_{t_j}^{(n-1)})}\bigg)=\bigg(S{_\l\hat
x_{t_j}^{(n-1)}}+f_1(_\l\hat
x_{t_j}^{(n-1)},_\l\hat y_{t_j}^{(n-1)})\bigg)_{\l=1,\ldots,r}\label{eq:f},\\
\a_i^{(n)}&=\frac{1}{\sqrt{r}}(\bar{\A_i})^\dagger({\textbf{x}}_0-\sum_{j=i}^{N-1}{_\l\textbf{f}}(t_j)^{(n-1)}\Delta_j)\nonumber\\
&=\frac{1}{\sqrt{r}}(\bar{\A_i})^\dagger\left[ \left(
{\begin{array}{cr}
 _1x_0 & \\
 _2x_0 & \\
 \vdots & \\
 _{r}x_0 &
 \end{array} } \right)-\left( {\begin{array}{cr}
 \sum_{j=i}^{N-1}f(t_j, _1\hat x_{t_{j}}^{(n-1)}, _1\hat y_{t_{j}}^{(n-1)}) & \\
 \sum_{j=i}^{N-1}f(t_j, _2\hat x_{t_{j}}^{(n-1)}, _2\hat y_{t_{j}}^{(n-1)}) & \\
 \vdots & \\
 \sum_{j=i}^{N-1}f(t_j, _{r}\hat x_{t_{j}}^{(n-1)}, _r\hat y_{t_{j}}^{(n-1)}) &
 \end{array} } \right)\DD_j  \right]\label{eq:alphan},\\
%\hat x_{t_i}^{(n)}(\om)&=\sum_{d=1}^{D(i)}\a_{i,d}^{(n)}\eta_d^i(\om)\label{eq:x},\\
%\hat y_{t_i}^{(n)}(\om)&=e^{-Uh}[\hat
%y_{t_{i-1}}^{(n-1)}(\om)+f_2(\hat x_{t_{i-1}}^{(n-1)}(\om),\hat
%y_{t_{i-1}}^{(n-1)}(\om))\DD_j\\
%&+g_2(\hat
%x_{t_{i-1}}^{(n-1)}(\om),\hat y_{t_{i-1}}^{(n-1)}(\om))\DD
%W_j(\om)]\label{eq:y},\\
(_\l\hat {y}_{t_i}^{(n)})_{\l=1,\ldots,r}&\triangleq(_1\hat
y_{t_i}^{(n)}, _2\hat
y_{t_i}^{(n)}, \cdots, _{r}\hat y_{t_i}^{(n)})'\nonumber\\
&=\bigg(e^{-Uh}[_\l\hat y_{t_{i-1}}^{(n-1)}+f_2(_\l\hat
x_{t_{i-1}}^{(n-1)},_\l\hat
y_{t_{i-1}}^{(n-1)})\DD_j\nonumber\\
&+g_2(_\l\hat x_{t_{i-1}}^{(n-1)},_\l\hat y_{t_{i-1}}^{(n-1)})\DD
_\l W_j]\bigg)_{\l = 1,\ldots, r}\label{eq:yn}.
\end{align}
where $_\l\hat {\textbf{x}}_{t_i}^{(n-1)}\triangleq(_1\hat
x_{t_i}^{(n-1)}, \ldots, _{r}\hat x_{t_i}^{(n-1)})$ are independent
copies of $\hat x_{t_i}^{(n-1)}$ corresponding to independent copies
of basis functions $_\lambda\eta_d^{i}$.

Upon converge, i.e.
\begin{align}
\max_i\{\EX\mid x_{t_i}^{(n)}-x_{t_i}^{(n-1)}\mid^2+\EX\mid y_{t_i}^{(n)}-y_{t_i}^{(n-1)}\mid^2<tol\}\label{conv},
\end{align}
we have $y_{t_N}^{(n)}$ as our approximation of $y_0$.

\begin{remark}
Although we only need the final grid point value $y_{t_N}$ (recall that $t_N = 0$) as our approximation of the inertial manifold $y_0:=\Phi_T(x_0)$, the intermediate points $(x_{t_i}^{(n)},y_{t_i}^{(n)})$ are
approximated by $(\hat x_{t_i}^{(n)}, \hat y_{t_i}^{(n)})$ as follows
\begin{align}
\hat x_{t_i}^{(n)}(\om)&=\sum_{d=1}^{D(i)}\a_{i,d}^{(n)}\eta_d^i(\om),\label{eq:x}\\
\hat y_{t_i}^{(n)}(\om)&=e^{-Uh}[\hat y_{t_{i-1}}^{(n)}(\om)+f_2(\hat
x_{t_{i-1}}^{(n-1)}(\om),\hat y_{t_{i-1}}^{(n)}(\om))\DD_j+g_2(\hat
x_{t_{i-1}}^{(n-1)}(\om),\hat y_{t_{i-1}}^{(n)}(\om))\DD W_j(\om)]\label{eq:y}.
\end{align}
Therefore, this approach of approximating stochastic inertial manifold also provides a way of solving backward-forward stochastic differential equations.
As in equation \eqref{eq:x}, $x_{t_i}^{(n)}(\om)$'s are approximated
by its orthogonal projection $\hat
x_{t_i}^{(n)}(\om):=\sum_{d=1}^{D(i)}\a_{i,d}^{(n)}\eta_d^i(\om)$,
where the randomness comes from the basis functions $\eta_d^i$
rather than the coefficients $\a_{i,d}^{(n)}$. The fact that $\a$'s
are deterministic can be seen from \eqref{eq5}, since $\a$'s are
expectations. In the numerical simulation of $\a$'s, we utilize
copies of $(x,y)$'s (different copies corresponding to different
sample paths $\om$) to calculate the expectations.

\end{remark}

%Note that here we fix $\omega$ in , $d = 1,\ldots, D(i)$, but the
%calculation of $\a_{i,d}^{(n)}$ using the Monte Carlo method
%requires all $\omega\in\Omega$ in principal and thus is
%deterministic.

% We choose $e_d^i=W_d-W_{d-1}, d=1,\ldots,i$, in our
%numerical examples in the next section.

All the conditions required by Bender and Denk in \cite{Bender} are
satisfied in our case, so the error analysis results for the Monte
Carlo simulation also apply here.

% With the
%above procedure, we compute $x_{t_i}^{(n)}, y_{t_i}^{(n)}$ in
%\eqref{eq8}. Thus we obtain the approximate stochastic inertial
%manifold $\Phi_T$.

Figure \ref{flowchart}  demonstrates the procedure of computing
stochastic inertial manifold.

\begin{figure}[!htb] %\label{BMpath}
\centering
 \includegraphics[scale=0.8]{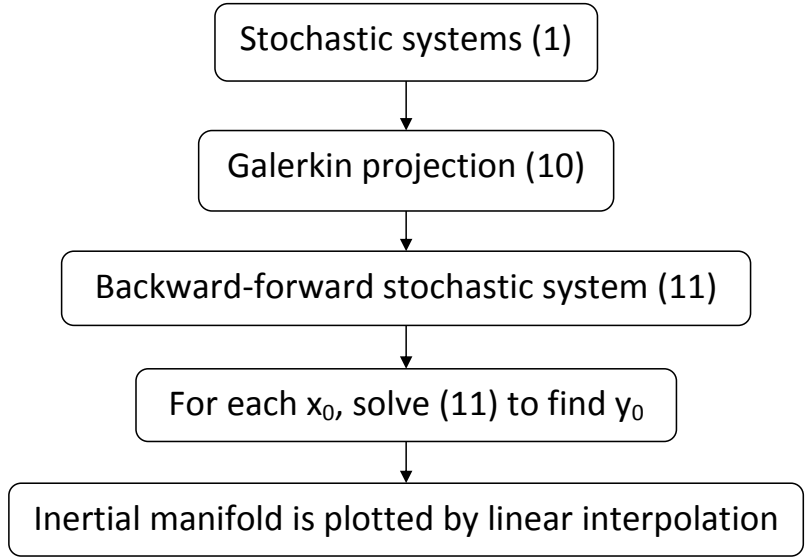}
\caption{\small \sl Procedure for computing a stochastic inertial
manifold.}\label{flowchart}
\end{figure}

Specifically, the computation is achieved in the following way as shown in Figure~\ref{schema}: We
begin with the flat manifold, i.e. let all the $y$-values be zero. In
principle, we could take any other acceptable initial manifold
and let it flow forward, and at the limit it would also approach the
desired manifold. We thus start our iteration by letting the initial
guess $y$ be zero at each $t_i$, see \eqref{eq:y0}. In order to avoid
nestings of conditional expectations, we also let initial guess $x$
values to be zero except at terminal time $t_N$, which corresponds
to letting all but the final set of coefficients $\a$ to be $0$, see \eqref{eq:alpha0}.
 Recall that we approximate each $x_{t_i}^{(n)}$
value by its finite dimensional orthogonal projection as in
\eqref{xxx}, so indeed we want to calculate $\a_i^{(n)}$. The
terminal value of $x$ is set to be a $\F_0-$measurable random
variable. At each iteration, in order to update
$\a_{i-1}^{(n)}$~\eqref{eq:alphan}, we use copies of $_\l
x_{t_{i-1}}^{(n-1)},\cdots, _\l x_{t_{N}}^{(n-1)}$ as well as $_\l
y_{t_{i-1}}^{(n-1)},\cdots, _\l y_{t_{N-1}}^{(n-1)}$ ($\l =
1,\ldots, r$) from previous iteration, and then generate copies of
$x_{t_{i-1}}^{(n)}$ by virtue of copies of basis functions $\eta$'s, see \eqref{eq:xrandom}.
Copies of $y_{t_{i-1}}^{(n)}$ are reproduced by an Euler-Maruyama scheme as in
\eqref{eq:yn}, in which different copies correspond to the sample
$\om$'s that are already chosen in basis functions $\eta$'s. This procedure
is repeated until $(x,y)$ converges in the mean square sense
\eqref{conv}. We finally acquire the terminal value of $y$ at the
stopped iteration as the approximation of $y_0$.

\begin{figure}[!htb] %\label{BMpath}
\centerline{\scalebox{0.7}{
 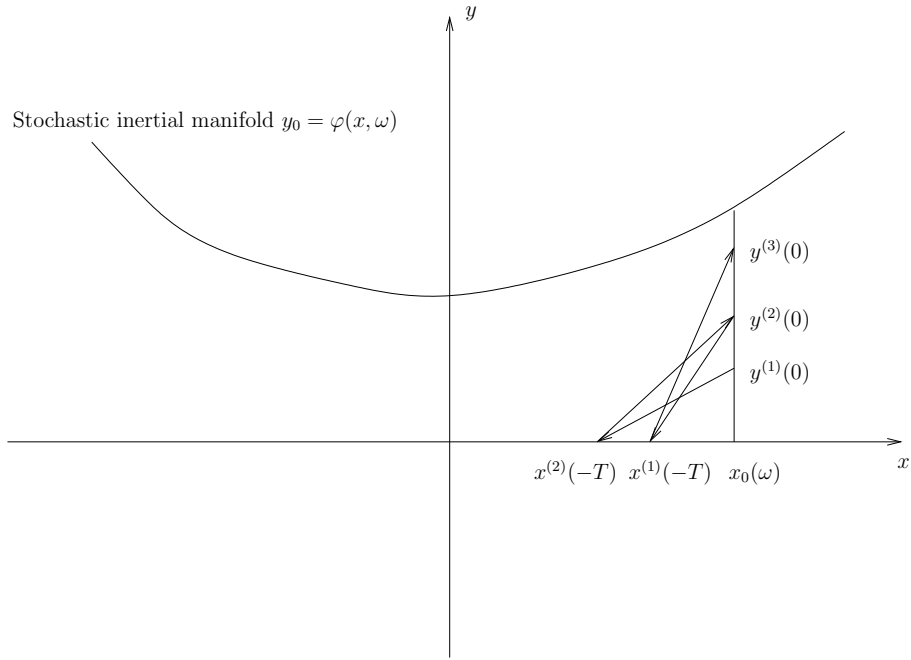}}
\caption{\small \sl Schematic diagram of the numerical
iteration.}\label{schema}
\end{figure}

\section{Examples}

In this section, we test our backward-forward numerical scheme in
two  examples. One is a system of \textsc{sde}s, and the other one is an \textsc{spde} (which is converted to a system of \textsc{sde}s).

\noindent\textbf{Example 1:  A system of stochastic ordinary differential equations}

\begin{align}
dX_t &= ( - aX_t-X_tY_t)dt,   \label{eq999} \\
dY_t &= (-Y_t(1+2Y_t)^+ + X_t^2)dt + \sigma Y_t \circ dW_t,   \label{eq999abc}
\end{align}
where $W_t$ is a scalar Wiener process, $a$ and $\sigma$ are real
parameters, $\circ$ indicates the Stratonovich interpretation of the
noise term and $f^+\triangleq\max(f,0)$.

 Figure \ref{eg1pplane}  is the phase portrait for the deterministic counterpart of
Example 1 ($\sigma=0$).
\begin{figure}[!htb] %\label{BMpath}
\centerline{
 \includegraphics[scale=0.5]{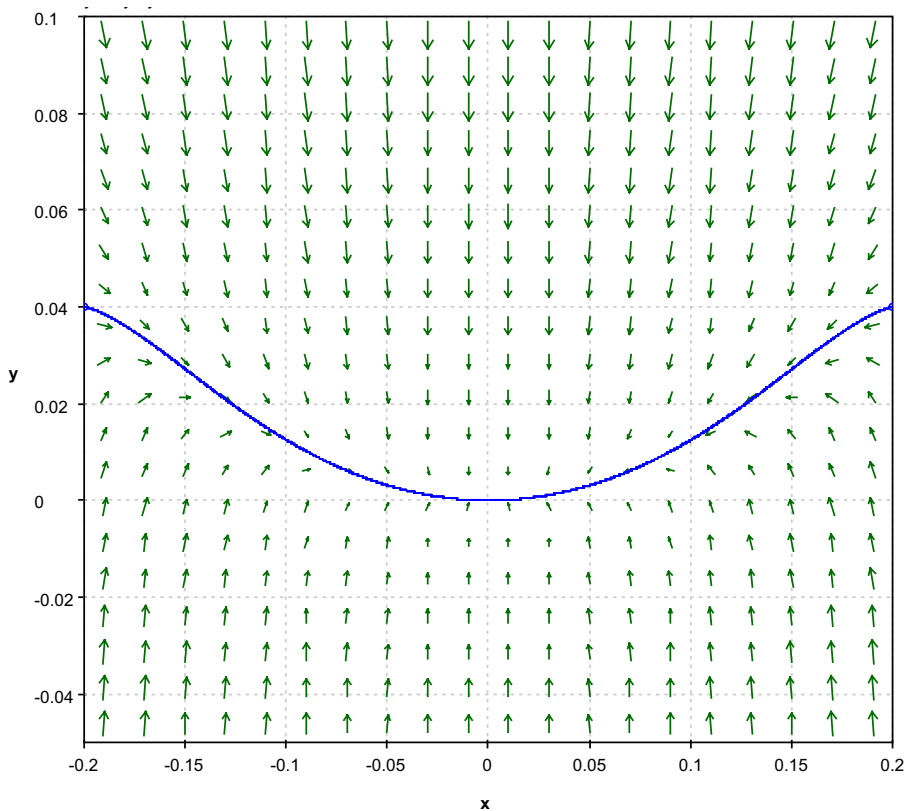}}
\caption{\small \sl Slow manifold for the deterministic counterpart
of Example 1: $\sigma=0$ and $a=0.1$.}\label{eg1pplane}
\end{figure}

Take $\sigma=0.1$ and $a=0.1$. Figure \ref{eg1tra} shows several
sample solution paths  of the \textsc{sde} system \eqref{eq999}-\eqref{eq999abc}.

%from different initial conditions (0.1,-0.5),(-0.1,-0.5),(0.4,-0.4)
%(-0.4,-0.4),(0.2,0.4),(-0.2,0.4),(0.4,0.8),(-0.4,0.8) for different realizations or the noise

\begin{figure}[!htb] %\label{BMpath}
\centering
 \includegraphics[scale=0.6]{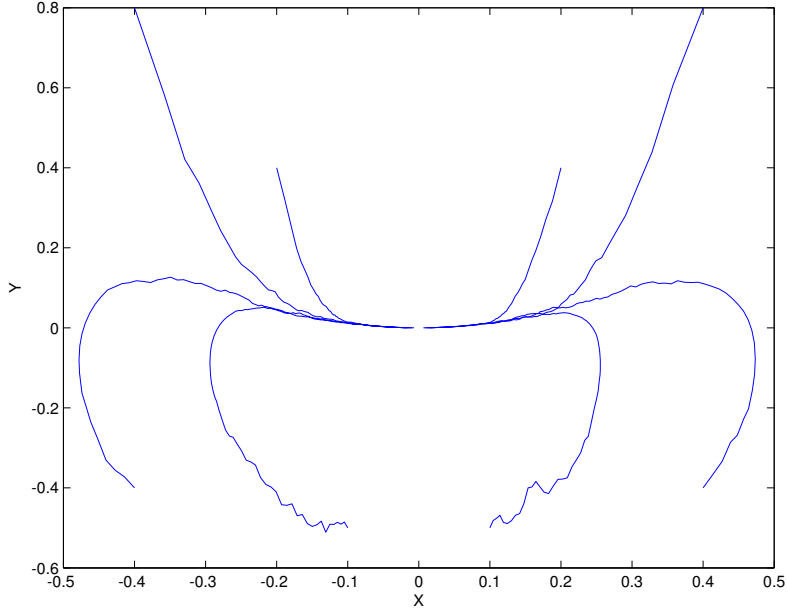}
\caption{\small \sl A few solution paths of Example 1: $a=0.1$ and
$\sigma=0.1$.}\label{eg1tra}
\end{figure}

Roberts \cite{Roberts} introduced a normal form transform method for
 stochastic differential systems with both slow modes and quickly decaying modes. The (approximate)
formula for the slow manifold (a type of inertial manifold) of
Example 1 obtained via his method
\footnote{\url{http://www.maths.adelaide.edu.au/anthony.roberts/sdesm.html}}.
is
\begin{align}
y \approx x^2 + 0.1x^2\int_{-\infty}^te^{-(t-\tau)}dW_\tau(\om)
\end{align}
and
\begin{align}
x \approx x + 0.1x\int_{-\infty}^te^{-(t-\tau)}dW_\tau(\om).
\end{align}
In Figure~\ref{eg1slowmfd}, we plot the stochastic inertial manifold according to
$y\approx x^2 + 0.1x^2\int_{-50}^0e^{\tau}dW_\tau$ and compare it
with the stochastic inertial manifold from our backward-forward
method for one sample path $\omega$. There is a remarkable agreement of the
shapes of the stochastic slow manifold obtained by the two distinct methods.
For four samples in Figure \ref{real},  Figure~\ref{difference} then shows the discrepancy between the stochastic inertial manifold  realised on three realisations $\bar{\om}, \tilde{\om}, \hat{\om}$ and that realised on $\om$ via our backward-forward approach, respectively, i.e. $M(\bar{\om})-M(\om)$, $M(\tilde{\om})-M(\om)$, and $M(\hat{\om})-M(\om)$ in red, blue and green color.

%Figure~\ref{real} gives the corresponding sample paths of simulations as in %Figure~\ref{difference} in color magenta, red, blue and green.

%\begin{table}[h]
%\begin{tabular}{cc}
%\includegraphics[height=2.6in,width=2.5in]{slowmfd.pdf}&\includegraphics[height=3.7in,width=2.5in]{difference.pdf}
%\end{tabular}
%\end{table}

\begin{figure}[!htb] %\label{BMpath}
\centering
 \includegraphics[scale=0.6]{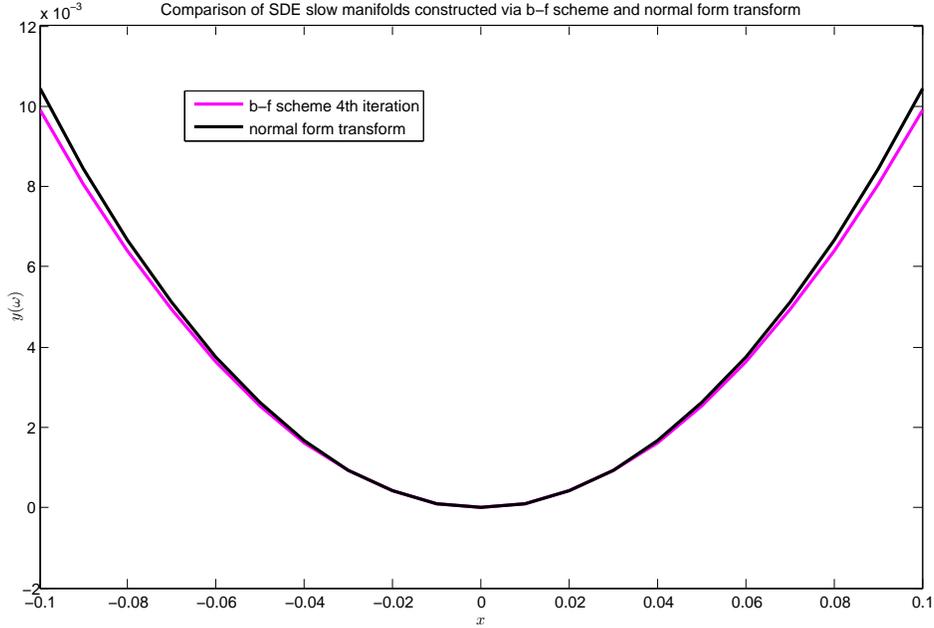}
\caption{\small \sl  Stochastic inertial manifold $\mathcal{M}
(\om)$ calculated by the backward-forward method and by the normal
form transform method for Example 1.
Online version: Magenta from the backward-forward method and black from
the normal form transform method.  }\label{eg1slowmfd}
\end{figure}

\begin{figure}[!htb] %\label{BMpath}
\centering
 \includegraphics[scale=0.6]{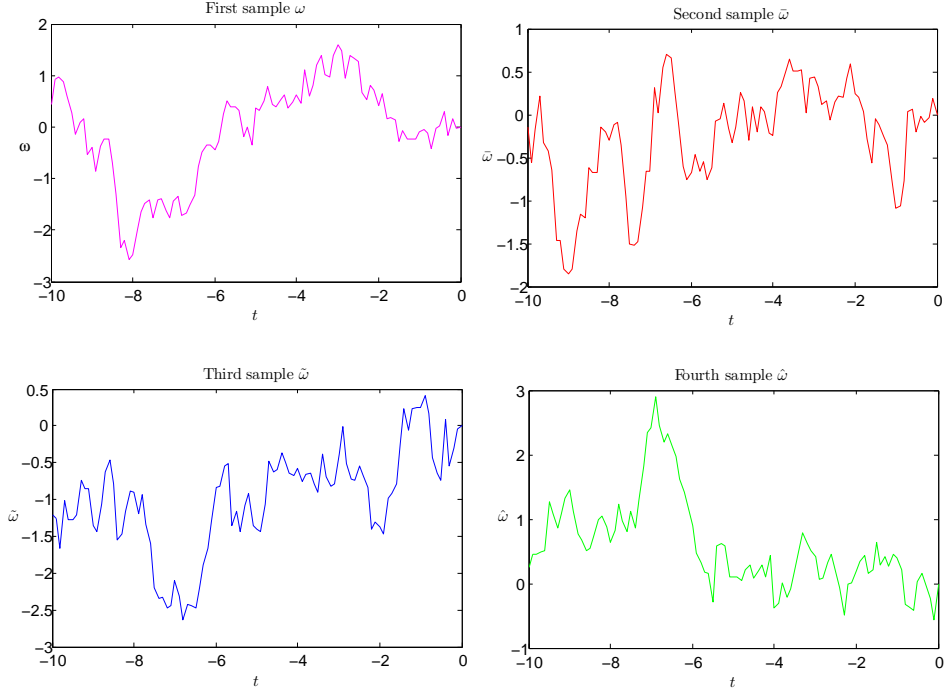}
\caption{\small \sl Four samples \textcolor{magenta}{$\omega$}, \textcolor{red}{$\bar{\omega}$}, \textcolor{blue}{$\tilde{\om}$} and \textcolor{green}{$\hat{\om}$} for Example 1.}\label{real}
\end{figure}

\begin{figure}[!htb] %\label{BMpath}
\centering
 \includegraphics[scale=0.6]{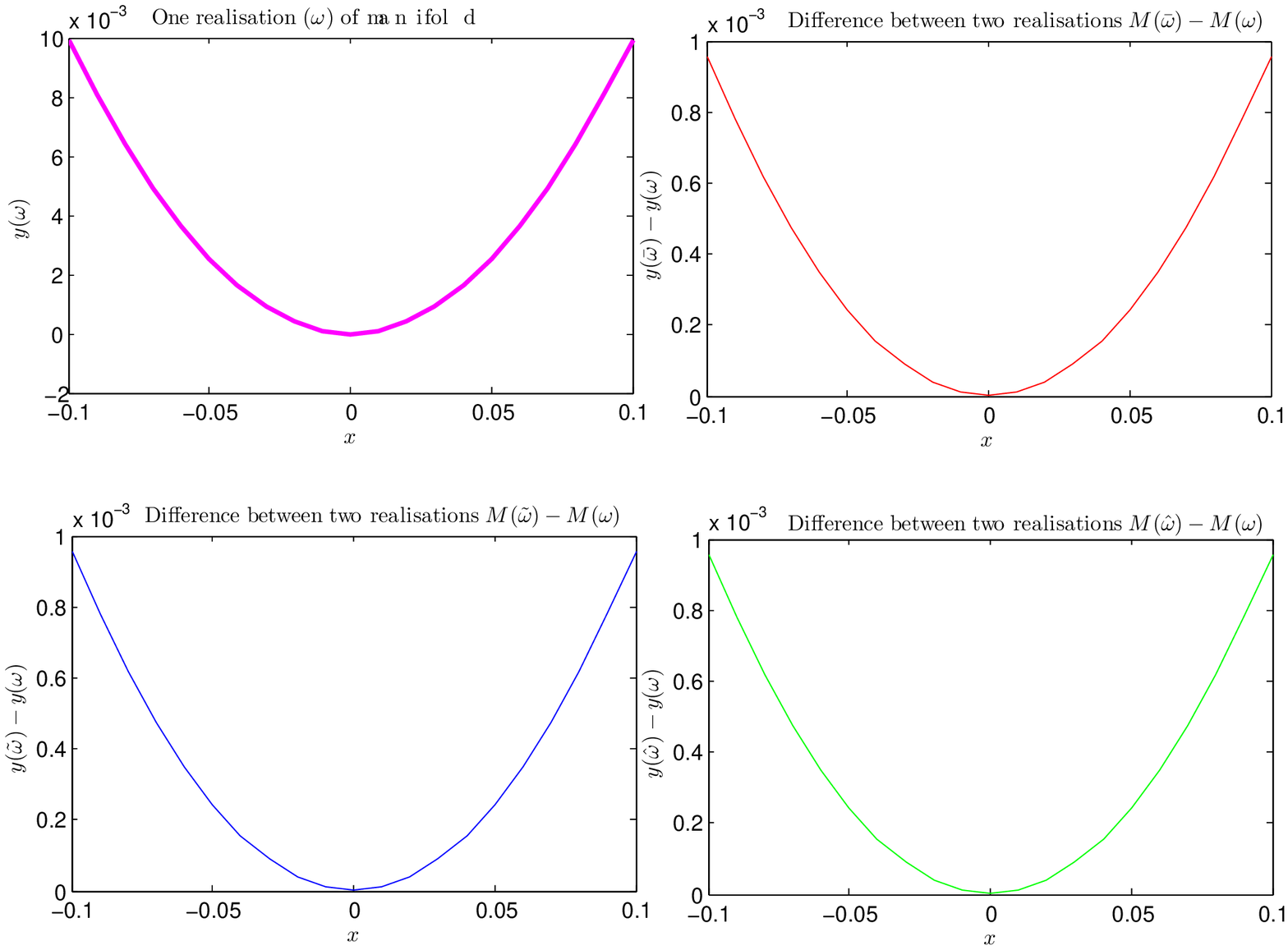}
\caption{\small \sl Discrepancy between different realisations of  the stochastic inertial manifold in Example 1 via the backward-forward method. Online version: Magenta (top left) for the stochastic manifold realised on sample $\textcolor{magenta}{\om}$, while red (top right), blue (bottom left) and green (bottom right) for discrepancies between stochastic manifold realised on samples $\textcolor{red}{\bar{\om}}$, $\textcolor{blue}{\tilde{\om}}$ and $\textcolor{green}{\hat{\om}}$ and stochastic manifold realised on sample $\textcolor{magenta}{\om}$. }\label{difference}
\end{figure}

\bigskip

\noindent\textbf{Example 2: A stochastic partial differential
equation }

We examine how effective the stochastic inertial manifold approach
is in assisting the simulation of the long-term dynamics of an
\textsc{spde}. We consider a stochastic partial differential
equation
\begin{align}
du_t&=\nu \partial_{xx} u_tdt+(u_t-u_t^3)dt+u_tdW_t, \quad 0<x<1, \label{spde}  \\
u_t(0)&=u_t(1)=0, \notag\\
u_0(x)&=\sqrt 2 \sin(\pi x) \notag,
\end{align}
where $\nu=0.01$, $W_t$ is a scalar Wiener process, and It\^{o} interpretation of the noise term is
adopted.

Let us first consider the dimension of the inertial manifold. Note that the
eigenvalues of the operator $\nu\partial_{xx}+I$ are $-\nu i^2\pi^2+1$ with the corresponding eigenmodes $e_i(x)=\sqrt{2} \sin(i \pi x)$, for $i=1, 2,\cdots$. Thus the dimension of the deterministic unstable
eigenspace is 3. In this example, $H=L^2(0, 1)$. Let $P$ be the orthogonal projection to the (deterministic) unstable eigenspace $H^+$, spanned by  eigenmodes $e_i, i=1, 2, 3$.

The existence of the stochastic inertial manifold requires the
nonlinear terms to be the globally Lipschitz. Here, in Example 2,
although the drift term $f(u):=u-u^3$ is only locally Lipschitz, we
can prepare the equation by replacing $f$ by the cutoff function
$\tilde f$ which is defined to be $f$ in a bounded neighborhood
centered at the origin, and 0 otherwise. The details of this
procedure were described by Da Prato and Debussche \cite{DaPrato}.

With the Galerkin projection $u=\sum_i u_i(t) e_i(x)$,  the evolutionary equations we will be working on become
\begin{align*}
\tilde p:\begin{cases}
&du_1(t)=\nu\l_1u_1(t)dt + \langle  u-u^3, e_1\rangle dt + u_1(t)dW_t,\\
&du_2(t)=\nu\l_2u_2(t)dt + \langle u-u^3,e_2\rangle dt + u_2(t)dW_t,\\
&du_3(t)=\nu\l_3u_3(t)dt + \langle u-u^3,e_3\rangle dt +
u_3(t)dW_t,
\end{cases}
\end{align*}
\[
\tilde q:du_4(t)=\nu\l_4u_4(t)dt + \langle u-u^3,e_4\rangle dt +
u_4(t)dW_t,
\]
where we only keep four Fourier modes $u_i(t)e_i, i = 1, 2, 3, 4$ for $u$; $u_5(t), u_6(t),\ldots$ are
simply ignored, as numerical simulations indicate that they are negligible in this specific case.

 %The following diagram shows the approximation
%procedure.

%\begin{displaymath}
%\begin{array}{cccccccccccc}
%\text{time}& & & t & t_0 & t_1 & t_2 & \cdots & t_{N-2} & t_{N-1} & t_{N}\\
%           & & &   &     &     &     & \vdots &\vdots &\vdots & &\\
% \text{$n$-th backward step} & \overleftarrow{x} & & x_1 & \widetilde{\circ} & \circ & \circ & \cdots & \circ & \widehat{\circ} & \widehat{\bullet} \\
%& & & x_2 & \widetilde{\circ} & \circ & \circ & \cdots & \circ & \widehat{\circ} & \widehat{\bullet} \\
%& & & x_3 & \widetilde{\circ} & \circ & \circ & \cdots & \circ & \widehat{\circ} & \widehat{\bullet} \\
% \text{$n$-th forward step} & \overrightarrow{y} & & & \star & \ast & \ast & \cdots & \ast & \ast &
% \widehat{\ast}\\
%\end{array}
%\end{align*}
%$n$-th iteration
%\begin{align*}
%\begin{array}{cccccccc}
 % \text{$n+1$st backward step} & \overleftarrow{x} & & x_1 & \circ & \circ & \circ & \cdots & \circ & \overline{\circ} & \bullet \\
 % & & & x_2 & \circ & \circ & \circ & \cdots & \circ & \overline{\circ} & \bullet \\
 % & & & x_3 & \circ & \circ & \circ & \cdots & \circ & \overline{\circ} & \bullet \\
 % \text{$n+1$st forward step} & \overrightarrow{y} & & & \widetilde{\star} & \underline{\ast} & \ast & \cdots & \ast & \ast &
 % \ast\\
 % & & & & & \vdots &\vdots &\vdots & &
%\end{array}
%\end{displaymath}

It is difficult to visualize the  stochastic
inertial manifold directly. But we can plot a ``point" $p$ in $H^+$ and the corresponding ``point" $q = \Phi(p,\om)e_4(x)$ on the inertial manifold, separately.  The ``point" $p$ may be represented as
$p \triangleq\sum_{i=1}^3 u_i(0,\omega) e_i(x)$ with coordinates~$ (u_1(0,\om),u_2(0,\om),u_3(0,\om))$. The corresponding ``point"  $ q  $ on the stochastic inertial manifold is then computed through our backward-forward approach.

%We denote the corresponding point on the inertial manifold by  $q \triangleq\Phi(p,\om)$.

We plot three different realizations of a point $p$ versus space variable
$x$ in the left panel of Figure~\ref{fig:spde}, and then plot the
corresponding point $q$ versus space variable $x$ separately in the right panel.

%examine the efficiency of our backward-forward method by comparing
%the numerical solution resulted
%from the inertial manifold reduction with that obtained by direct numerical
%simulations of the original stochastic partial differential equation.

%The solution via the inertial manifold reduction is $u=x+\Phi(x)$ where $\Phi$ is the graph for the inertial manifold.  The direct numerical simulation of the solution $u$  is obtained via, for example, a method in \cite{Jentzen}. The stochastic inertial manifold is of dimension 30, while for direct numerical simulation (\cite{Jentzen}), we use 100-dimensional Ito-Galerkin truncation to
%approximate the solution $u$ of the original stochastic partial differential equation \eqref{spde}. We observe that the
%  solution obtained by two methods are close to each other; see Figure
 % \ref{spde-solu}.

\begin{figure}[!htb]
\begin{tabular}{cc}
 \includegraphics[scale=0.5]{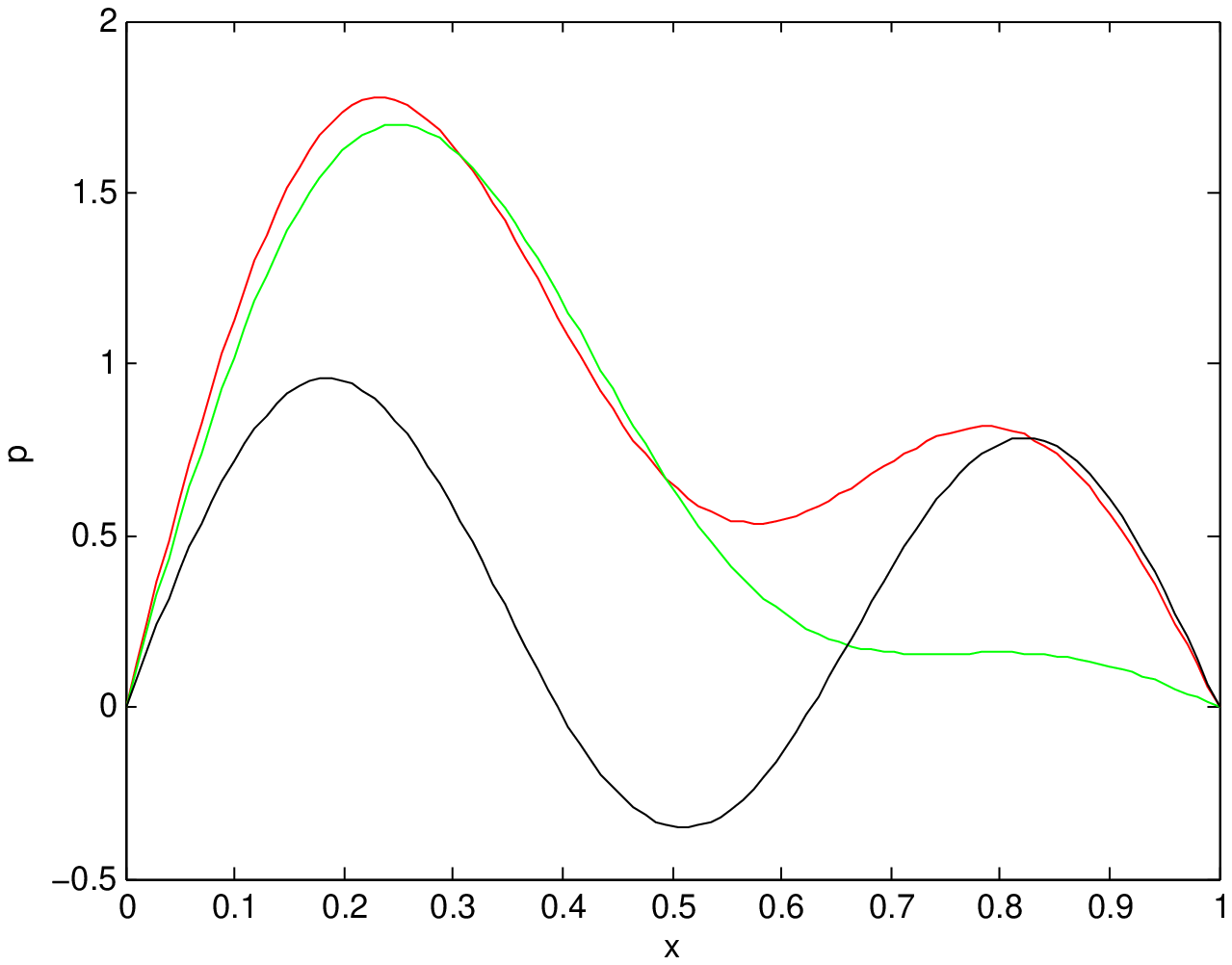}&\includegraphics[scale=0.5]{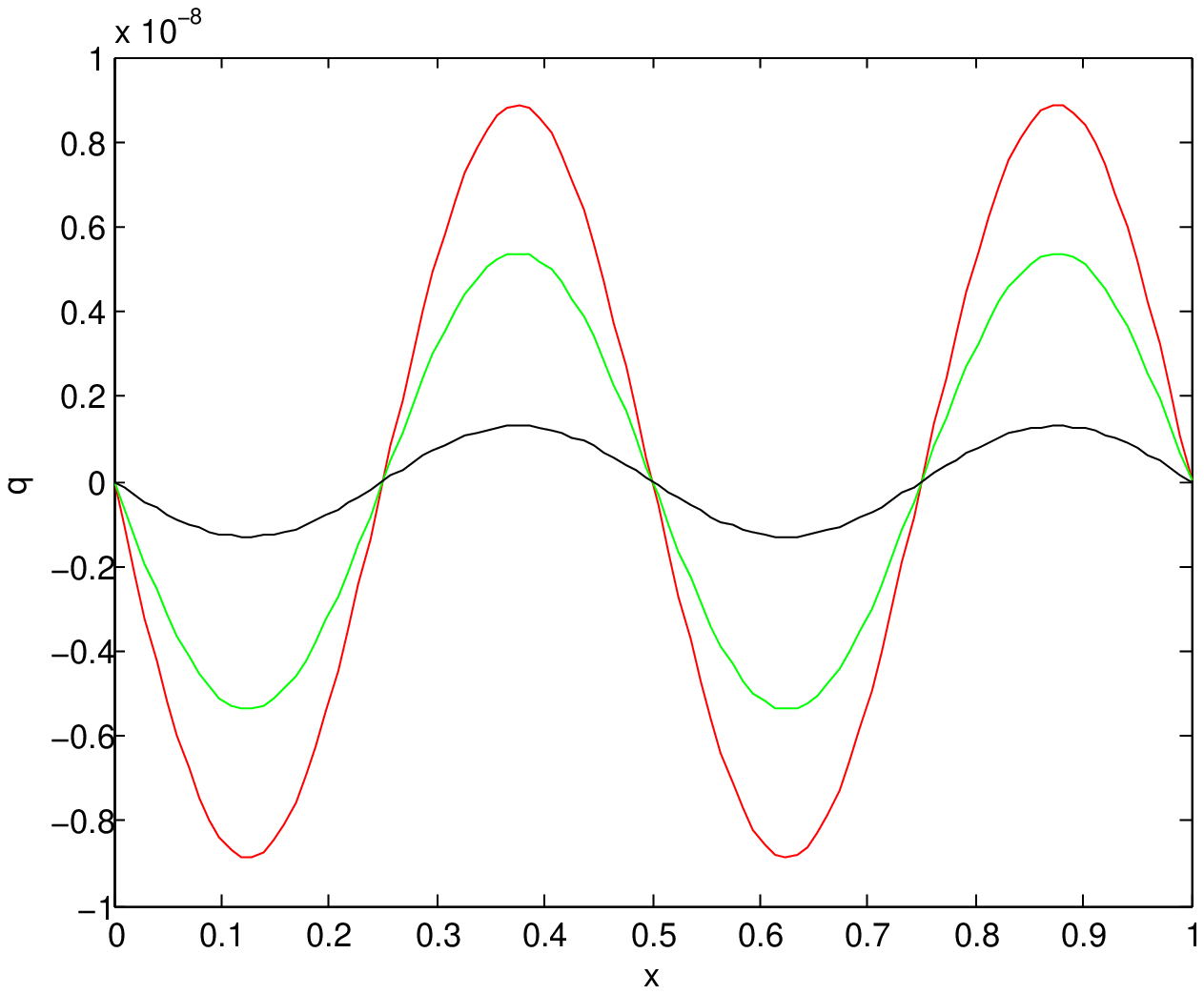}
 \end{tabular}
\caption{\small \sl  Example 2 -- A ``point" $p=u_1(0,\om)   e_1(x)+u_2(0,\om) e_2(x)+u_3(0,\om) e_3(x)$, with $u_1, u_2$
and $u_3$ being uniformly distributed random numbers, in  
$H^+$ and the corresponding ``point" $q$ in the stochastic inertial manifold: Three samples of $p$ (left panel) and the corresponding three samples of $q$ (right panel).
 Online version: Corresponding $p$ and $q$ samples
  have the same color (green, red or black).}\label{fig:spde}
% Blue
%from direct numerical simulation and red from stochastic inertial
%manifold reduction. } \label{spde-solu}
\end{figure}

\section{Discussion and conclusion}

In this paper, we have devised a backward-forward numerical approach for computing stochastic inertial manifolds for
stochastic evolutionary equations, including higher dimensional stochastic ordinary differential equations and stochastic partial differential equations. This approach is based on the stochastic inertial manifold theory of Da Prato and Debussche \cite{DaPrato}, which requires a backward-forward approximation formulation. In fact, our approach also provides a stand-alone numerical scheme for solving backward-forward \textsc{sde}s.

Unlike deterministic evolutionary equations,  we need to guarantee the adaptedness  of solution processes with respect to an appropriate filtration for stochastic evolutionary equations. This makes the computation in the backward part of our numerical approach cumbersome, as it involves simulating conditional expectations which further requires a random basis. It would be nice to have a numerical scheme that uses only forward simulations. Clearly, a numerical scheme involving only forward simulations would be desirable; such schemes have been devised for deterministic problems (e.g. \cite{Gear, Zagaris}) and we are working on adapting them for the stochastic case.

\bigskip

%%%%%%%%%%%%%%%%%%%%%%%%
%%%%%%%%%%%%%%%%%%%%%%%%
%%%%%%     Acknowledgement
%%%%%%%%%%%%%%%%%%%%%%%%%%%%%
  {\bf Acknowledgements.}   We thank Arnaud Debussche,  Arnulf Jentzen, Edriss Titi, and  Jianfeng Zhang for helpful discussions.

%%%%%%%%%%%%%%
%%%%%%%%%%%%%%
%%%%%%%%%%%%%%

\end{document}